\documentclass[]{siamart220329}

\usepackage{lipsum}
\usepackage{amsfonts}
\usepackage{graphicx}
\usepackage{epstopdf}
\ifpdf
  \DeclareGraphicsExtensions{.eps,.pdf,.png,.jpg}
\else
  \DeclareGraphicsExtensions{.eps}
\fi
\usepackage[useregional]{datetime2}

\newsiamremark{remark}{Remark}
\newsiamremark{hypothesis}{Hypothesis}
\crefname{hypothesis}{Hypothesis}{Hypotheses}
\newsiamthm{claim}{Claim}

\headers{A Novel Gradient Methodology}{Varner and Patel}

\title{A Novel Gradient Methodology with Economical Objective Function Evaluations for Data Science Applications
\thanks
{
  Submitted to the editors \today.
  \funding{}
}
}

\author{
Christian Varner\thanks{Department of Statistics, University of Wisconsin Madison, Madison, WI 
  (\email{cvarner@wisc.edu}, \email{vivak.patel@wisc.edu}).}
\and 
Vivak Patel\footnotemark[2]}

\usepackage{amsopn}

\usepackage[utf8]{inputenc}
\usepackage{xcolor}
\usepackage{mathtools}
\usepackage{enumitem}
\usepackage{amsmath}

{
\theoremstyle{plain}
\newtheorem{assumption}{Assumption}
\newtheorem{property}{Property}
}

\usepackage{natbib}
\setcitestyle{numbers}

\usepackage[para,online,flushleft]{threeparttable}
\usepackage{amsfonts}
\usepackage{amssymb}
\usepackage{float}
\usepackage{algpseudocode}
\usepackage{subfiles}
\usepackage{indentfirst}
\usepackage{graphicx}
\usepackage{tabularx}
\usepackage{caption}
\usepackage{array,booktabs}
\usepackage{subcaption}
\usepackage{ragged2e}
\usepackage{tcolorbox}
\tcbuselibrary{breakable}
\tcbset{
  width=0.99\textwidth,
  halign=justify,
  center,
  breakable,
  colback=white    
}
\usepackage{tikz}
\usetikzlibrary{plotmarks}
\usetikzlibrary{patterns}
\usepackage{pgfplotstable}
\usepackage{pgfplots}
\pgfplotsset{compat=1.15}
\usepgfplotslibrary{statistics}
\pgfplotsset{boxplot legend/.style={
    legend image code/.code={
        \draw[#1] (0cm,0cm) rectangle (0.6cm,0.3cm)
        (0.3cm,0cm) -- (0.3cm,-0.1cm) (0.1cm,-0.1cm) -- (0.5cm,-0.1cm)
        (0.3cm,0.3cm) -- (0.3cm,0.4cm) (0.1cm,0.4cm) -- (0.5cm,0.4cm);
    },
}}
\newcounter{iloop} 

\usepackage{pifont}

\usepackage{caption}
\usepackage[export]{adjustbox}
\usepackage[margin=1in]{geometry}

\crefname{assumption}{assumption}{assumptions}
\crefname{property}{property}{properties}
\crefname{definition}{definition}{definitions}
\crefformat{footnote}{#2\footnotemark[#1]#3}

\newcommand{\defeq}{\vcentcolon=}
\newcommand{\eqdef}{=\vcentcolon}
\newcommand{\inlinenorm}[1]{ \Vert #1 \Vert}
\newcommand{\norm}[1]{\left\Vert #1 \right\Vert}

\newcommand{\bigpar}[1]{\left( #1 \right)}

\DeclareMathOperator{\obj}{obj}
\DeclareMathOperator{\gra}{grad}
\DeclareMathOperator{\low}{lower}
\DeclareMathOperator{\upp}{upper}
\DeclareMathOperator{\iter}{iter}
\DeclareMathOperator{\exit}{exit}

\DeclareMathOperator{\true}{true}

\DeclareMathOperator{\reference}{ref}

\begin{document}

\maketitle

\begin{abstract}
Gradient methods are experiencing a growth in methodological and theoretical developments owing to the challenges posed by optimization problems arising in data science. However, such gradient methods face diverging optimality gaps or exploding objective evaluations when applied to optimization problems with realistic properties for data science applications.
In this work, we address this gap by developing a generic methodology that economically uses objective function evaluations in a problem-driven manner to prevent optimality gap divergence and avoid explosions in objective evaluations. 
Our methodology allows for a variety of step size routines and search direction strategies. Furthermore, we develop a particular, novel step size selection methodology that is well-suited to our framework.
We show that our specific procedure is highly competitive with standard optimization methods on CUTEst test problems.
We then show our specific procedure is highly favorable relative to standard optimization methods on a particularly tough data science problem: learning the parameters in a generalized estimating equation model.
Thus, we provide a novel gradient methodology that is better suited to optimization problems from this important class of data science applications.
\end{abstract}



\section{Introduction} \label{sec:Introduction}
Gradient methods are experiencing a growth in methodological and theoretical developments in order to address the needs of data science applications, in which full objective function and gradient function evaluations are expensive to compute \cite[][\S 3]{bottou2018optimization}. For example, gradient methods that never make (or make pre-specified) objective function evaluations are being rejuvenated and actively developed 
(Gradient descent with diminishing step-size \cite{bertsekas1999,josz2023,patel2023gradient};
Gradient descent with constant step-size \cite{armijo1966minimization,zhang2020gradientclipping,josz2023,li2023convex};
Barzilai-Borwein Methods \cite{barzilai1988,burdakov2019stabilized};
Nesterov's Acceleration Method \cite{nesterov2012GradientMF,li2023convex};
Bregman Distance Methods \cite{bauschke2017descentlemma};
Negative Curvature Method \cite{curtis2018exploiting};
Lipschitz Approximation  \cite{malitsky2020adaptive,malitsky2023adaptive};
Weighted Gradient-Norm Damping \cite{wu2020wngrad,grapiglia2022AdaptiveTrust}; 
Adaptively Scaled Trust Region \cite{gratton2022firstorderOFFO,gratton2023OFFOSecondOrderOpt};
Polyak's Method \cite{polyak1969}).
Such gradient methods generally enjoy inexpensive per-iteration costs, and global convergence guarantees when the gradient function is globally Lipschitz smooth---that is, the gradient function is Lipschitz continuous with a common rank for any compact set in the domain.\footnote{Such objective functions are referred to as Lipschitz smooth in the literature. Here we will refer to such functions as globally Lipschitz smooth to distinguish them from the relevant locally Lipschitz smooth case, in which the rank can depend on the choice of compact set.}

Unfortunately, the globally Lipschitz smoothness condition does not apply to smooth objective functions arising in data science problems, such as learning for a feed forward network \cite{patel2022globalconvergence}, learning for a recurrent neural network \cite{patel2022globalconvergence}, factor analysis for pattern recognition, inverse Gaussian regression, and estimating parameters in generalized estimating equations \cite[see][\S 2]{varner2024challenges}.
To exacerbate this issue, for objective functions satisfying the realistic locally Lipschitz smoothness condition \cite[i.e., the aforementioned examples; see][Table 1.1]{varner2024challenges}, \textit{all} the aforementioned gradient methods are shown to generate iterates with diverging optimality gaps \cite[see][Table 1.2 and \S 3]{varner2024challenges}.

From a reliability perspective, gradient methods that make use of objective function evaluations to ensure descent seem more attractive, as they naturally prevent such divergence 
(Armijo's Backtracking Method \cite{armijo1966minimization,zhang2020firstorder};
Newton's Method with Cubic Regularization \cite{nesterov2006cubic};
Lipschitz Constant Line Search Methods \cite{nesterov2012GradientMF,curtis2018exploiting}; 
Adaptive Cubic Regularization \cite{cartis2011cubic1,cartis2011cubic2}).\footnote{Such gradient methods can be shown to require only a finite number of accepted iterations to achieve a certain threshold for the gradient function, when the gradient function is globally Lipschitz continuous. Interestingly, such guarantees for this class of gradient methods appears to be readily extendable to the case where the gradient function is locally Lipschitz continuous \cite[Theorem 5]{zhang2020firstorder}. Furthermore, under the assumption of a globally Lipschitz continuous \textit{Hessian} function, these gradient methods seem amenable to results that control the number of objective function evaluations prior to an accepted iterate \cite[Theorem 2.1]{cartis2011cubic2}.}
However, for objective functions satisfying the locally Lipschitz smoothness condition, these gradient methods can grow exponentially in the number of objective function evaluations per accepted iterate \cite[see][Table 1.3 and \S 4]{varner2024challenges}. For our motivating class of important data science applications where objective evaluations are expensive yet gradient evaluations are inexpensive,\footnote{Learning the parameters in a generalized estimating equation model is the key application. A recent application of such problems include developing COVID policies \cite{woodruff2021riskfactors,bergwerk2021covid}. See \cite[Chapter 3]{hardin2012gee} for an introduction.} these gradient methods are infeasible because of the potential explosion in objective evaluation complexity.  

For such a class of problems, gradient methods that use objective evaluations in a problem-driven manner to prevent divergence and control objective evaluation complexity seem to be better choices.  
One promising gradient method does exactly this by switching between using a (nonmonotone) line search in some iterations and not using objective function evaluations in other iterations \cite{grippo1991class}. 
However, this method suffers in two ways. First, this method requires a rapid decay in the gradient function to avoid line search. This may not be possible as the method approaches a solution for a poorly conditioned problem, resulting in potentially many expensive objective function evaluations. Second, the method requires the compactness of the level set, which fails to hold, say, for linearly separating two perfectly separable sets of data points, as any scaling of the linear discriminant is a solution. In other words, even this promising gradient method falls short for important data science applications.

\textit{To summarize, existing gradient methods face computational or reliability challenges under realistic settings for optimization problems in important data science applications.} To begin addressing this gap, we introduce a novel framework that uses objective evaluations in an economical, problem-driven manner (see \cref{result-objective-function}), prevents divergence of the optimality gap (see \cref{result-gradient}), and allows for many procedures with different step size strategies and search direction strategies (e.g., Quasi-Newton directions). Our methodology achieves this by using novel, inexpensive tests on the behavior of the iterates to trigger when to evaluate the objective function, and uses this objective information to enforce descent with a generalization of Armijo's condition (see \cref{subsec-method}). Furthermore, we introduce a novel step size scheme that produces a highly effective procedure when paired with a negative gradient search direction (see \cref{sec:StepSize}). Through experiments, we show that our procedure is highly competitive against standard approaches on unconstrained optimization problems from the CUTEst suite (see \cref{sec:Results-cutest}). 
We then show that our procedure is dominant on optimization problems from our motivating data science applications (see \cref{sec:Results-gee}).
Consequently, to the best of our knowledge, we provide a general methodology---along with a novel, sophisticated step size procedure---that is practical and rigorously-justified for solving important optimization problems arising in data science applications. In turn, our methodology allows for the reliable and practical solution of a wide variety of optimization problems arising in data science.

\section{Problem Formulation} \label{sec:Problem}
Motivated by the properties of common data science problems, we aim to (locally) solve
\begin{equation} \label{problem}
\min_{\theta \in \mathbb{R}^n} F(\theta),
\end{equation}
where the objective function, $F: \mathbb{R}^{n} \to \mathbb{R}$, satisfies the following assumptions.
\begin{assumption} \label{as-bounded below}
The objective function, $F$, is bounded below by some constant $F_{l.b.} > -\infty$.
\end{assumption}

\begin{assumption} \label{as-loc-lip-cont}
    $\forall \theta \in \mathbb{R}^n$, the gradient function $\dot F(\theta) := \nabla F(\psi) \vert_{\psi=\theta}$ exists and is locally Lipschitz continuous.
\end{assumption}

For clarity, we define local Lipschitz continuity as follows.
\begin{definition} \label{def-lipschitz}
A function $G: \mathbb{R}^n \to \mathbb{R}^n$ is locally Lipschitz continuous if for every $\theta \in \mathbb{R}^n$, there exists an open ball of $\theta$, $\mathcal{N}$, and a constant $\mathcal L \geq 0$, such that $\forall \phi, \psi \in \mathcal{N}$,
\begin{equation} \label{eqn-lipschitz-bound}
\norm{ \dot G(\phi) - \dot G(\psi)}_2 \leq \mathcal L \norm{ \phi - \psi}_2.
\end{equation}
\end{definition} 
An equivalent definition for a function $G$ to be locally Lipschitz continuous is if for every compact set $C \subset \mathbb{R}^p$, there exists an $\mathcal L(C) \geq 0$ such that \cref{eqn-lipschitz-bound} holds for all $\phi, \psi \in C$.

\begin{tcolorbox}
\centering
We never assume that we have knowledge of the local Lipschitz rank in our gradient method.
\end{tcolorbox}

\section{Our General Method} \label{sec:Algo}
Motivated by the theoretical analysis technique developed in \cite{patel2022globalconvergence, patel2022stoppingcriteria,patel2023gradient} around triggering events, we present a novel gradient method that develops these theoretical triggering events into practical tests to monitor the behavior of the iterate sequence, determine when to check the objective function, and when to adapt algorithm parameters. 
In \cref{subsec-method}, we present our general method and important categories of variations such as the choice of the step size procedure.
In \cref{subsec-global-convergence}, we conduct a global convergence analysis of our general method.

\subsection{Our Method} \label{subsec-method}

\begin{algorithm}[!ht]
    \caption{Our General Method}
    \label{alg-general-algorithm}
    \begin{algorithmic}[1]
    \Require $F, \dot F, \theta_{0}, \epsilon > 0$    
    \Require $\mathrm{StepDirection()}$ \Comment{Cannot require additional objective evaluations; may be stateful}
    \Require $\mathrm{StepSize()}$ \Comment{Cannot require additional objective evaluations; may be stateful}
    \Require $\sigma_{\low} \in (0,1)$ \Comment{Reduction factor for step size scaling} 
    \Require $\sigma_{\upp} \geq 1$ \Comment{Increase factor for step size scaling}
    \Require $w \in \mathbb{N}$ \Comment{Number of objective values used in nonmonotone search}
    \Require $\rho \in (0,1)$ \Comment{Relaxation parameter in nonmontone Armijo's condition}
     
    \State $k \leftarrow 0$ \Comment{Outer loop counter}
   
    \State $\delta_k \leftarrow 1$ \Comment{Step size scaling}
      
    \State $\tau_{\obj}^0 \leftarrow F(\theta_0)$ \Comment{Nonmontone search threshold}
    \State Select $\tau_{\iter,\exit}^0, \tau_{\iter,\max}^0, \tau_{\gra,\low}^0, \tau_{\gra,\upp}^0$
    \Comment{Cannot require additional objective evaluations}
    
    \While{$\inlinenorm{ \dot F(\theta_k) } > \epsilon$} \Comment{Outer loop}
    
    \State $j, \psi_0^k \leftarrow 0, \theta_k$ \Comment{Inner loop counter and initialization}
    
    \While{$\true$} \Comment{Inner loop}
    		\State $\gamma_j^k, \alpha_j^k \leftarrow \mathrm{StepDirection()}, \mathrm{StepSize()}$
    		
    		\If{ $\inlinenorm{ \psi_j^k - \theta_k }_2 > \tau_{\iter,\exit}^k$ 
    			\textbf{or} $\inlinenorm{ \dot F(\psi_j^k)}_2 \not\in (\tau_{\gra,\low}^k, \tau_{\gra,\upp}^k)$ 
    			\textbf{or} $j == \tau_{\iter,\max}^k$} 
    			\If{ $F(\psi_j^k) \geq \tau_{\obj}^k + \rho \delta_k \alpha_0^k \dot F(\theta_k)^\intercal \gamma_0^k $ }
    			\Comment{Fails nonmonotone Armijo condition}
    				\State $\theta_{k+1}, \delta_{k+1} \leftarrow \theta_k, \sigma_{\low} \delta_k$
    				\Comment{Reset iterate with reduced step size scaling}
    				\State Select $\tau_{\iter,\exit}^{k+1}, \tau_{\iter,\max}^{k+1}, \tau_{\gra,\low}^{k+1}, \tau_{\gra,\upp}^{k+1}$
    				\Comment{Without more objective evaluations}

    			\ElsIf{ $\inlinenorm{ \dot F(\psi_j)}_2 \leq \tau_{\gra,\low}^k$ }
    				\State $\theta_{k+1}, \delta_{k+1} \leftarrow \psi_j^k, \delta_k$
				\Comment{Accept iterate and leave step size unchanged}    			
    				\State Set $\tau_{\obj}^{k+1} \leftarrow$ by \cref{eqn-nonmonotone-threshold}
				\State Select $\tau_{\iter,\exit}^{k+1}, \tau_{\iter,\max}^{k+1}, \tau_{\gra,\low}^{k+1}, \tau_{\gra,\upp}^{k+1}$
				\Comment{Without more objective evaluations}

    			\Else
    				\State $\theta_{k+1}, \delta_{k+1}, \leftarrow \psi_j^k, \sigma_{\upp}\delta_k$
    				\Comment{Accept iterate and increase step size}
    				\State Set $\tau_{\obj}^{k+1} \leftarrow$ by \cref{eqn-nonmonotone-threshold}

    				\State Select $\tau_{\iter,\exit}^{k+1}, \tau_{\iter,\max}^{k+1}, \tau_{\gra,\low}^{k+1}, \tau_{\gra,\upp}^{k+1}$
			    \Comment{No additional objective evaluations}

    			\EndIf
			\State $k \leftarrow k+1$
    			\State Exit Inner Loop
    		\EndIf
    		
    		\State $\psi_{j+1}^k, j \leftarrow \psi_j^k + \delta_k \alpha_j^k \gamma_j^k, j+1$
    		\Comment{Standard gradient-related method}
    
    \EndWhile 

    \EndWhile 
	\State \Return $\theta_k$    
    \end{algorithmic}
    
\end{algorithm} 

\Cref{alg-general-algorithm} is our novel, flexible framework for solving optimization problems satisfying \cref{as-bounded below,as-loc-lip-cont}. Importantly, \cref{alg-general-algorithm} allows for general choices of subroutines and parameters, which may be tailored to a particular problem. Furthermore, if these subroutines and parameters satisfy the sensible properties that we discuss next, we can provide a general and reasonable convergence analysis.

\paragraph{Step Direction} In \Cref{alg-general-algorithm}, the step direction at $\psi$ is generated by a procedure $\mathrm{StepDirection()}$, which---owing to our motivation of avoiding objective function evaluations---cannot make use of additional objective function evaluations and may be stateful. Because of possible statefulness, $\mathrm{StepDirection}()$ can generate two distinct step directions at a point $\psi$, if $\psi$ is visited at two distinct times. Thus, to control $\mathrm{StepDirection()}$, we require the following properties.
\begin{property} \label{property-stepdirection-negative}
For any compact set $C \subset \mathbb{R}^n$, $\exists \underline g(C) > 0$ such that for any $\gamma$ generated by $\mathrm{StepDirection()}$ at $\psi \in C$
satisfies $-\underline g(C) \inlinenorm{ \dot F(\psi)}_2^2 \geq \dot F(\psi)^\intercal \gamma$.
\end{property}
\begin{property} \label{property-stepdirection-propgrad}
For any compact set $C \subset \mathbb{R}^n$, 
  $\exists \overline g(C) > 0$ such that for any $\gamma$ generated 
  by $\mathrm{StepDirection()}$ at $\psi \in C$ 
  satisfies $\inlinenorm{ \gamma }_2 \leq \overline g(C) \inlinenorm{ \dot F(\psi)}_2$.
\end{property}

We now contextualize \cref{property-stepdirection-negative,property-stepdirection-propgrad} by using some specific examples. First, if $\mathrm{StepDirection}()$ is the negative gradient, then \cref{property-stepdirection-negative,property-stepdirection-propgrad} are satisfied with $\underline g(C) = \overline g(C) = 1$ for any compact $C \subset \mathbb{R}^n$. As another example, consider the case of an objective function that is twice-differentiable, strongly convex, and globally Lipschitz smooth; if the step direction is generated by Newton's method, then $\underline g(C)$ is the reciprocal of the Lipschitz rank, and $\overline g(C)$ is the reciprocal of the strong convexity constant for any compact $C \subset \mathbb{R}^n$ \cite[c.f.][Eq. 1.27]{bertsekas1999}. Thus, \cref{property-stepdirection-negative,property-stepdirection-propgrad} are generalizations of such important special cases, and allow for a broad and useful selection of step directions.\footnote{By \cref{property-stepdirection-propgrad}, if a first-order stationary point is found and accepted, then $\gamma$ generated at this point will be zero and the algorithm will terminate.}

\paragraph{Step Size} In \Cref{alg-general-algorithm}, $\mathrm{StepSize()}$ and $\lbrace \delta_k \rbrace$ determine the step size. As $\mathrm{StepSize()}$ may be stateful (c.f., $\mathrm{StepDirection()}$), we use the following properties.

\begin{property} \label{property-stepsize-upper}
For any compact set $C \subset \mathbb{R}^n$, $\exists \overline{\alpha}(C) > 0$ such that for any $\alpha$ generated by $\mathrm{StepSize()}$ at any $\psi \in C$ satisfies $\alpha \leq \overline\alpha(C)$. 
\end{property}
\begin{property} \label{property-stepsize-lower}
For any compact set $C \subset \mathbb{R}^n$, $\exists \underline{\alpha}(C) > 0$ such that for any $\alpha$ generated by $\mathrm{StepSize()}$ at any $\psi \in C$ satisfies $\alpha \geq \underline\alpha(C)$.
\end{property}

Trivially, \cref{property-stepsize-upper,property-stepsize-lower} are satisfied for constant step size procedures. \Cref{property-stepsize-upper,property-stepsize-lower} can be readily checked for a number of more sophisticated schemes in a variety of contexts. That is to say, \cref{property-stepsize-upper,property-stepsize-lower} appear to encompass a wide spectrum of step size selection schemes.

Finally, $\lbrace \delta_k \rbrace$ also contribute to the step size value. Accordingly, $\delta_{k+1}$ is either the same as $\delta_k$, can be reduced by a factor of $\sigma_{\low} \in (0,1)$, or increased by a factor of $\sigma_{\upp} \geq 1$. Specific values for these parameters are given in \cref{sec:StepSize}.

\paragraph{Nonmonotone Armijo Condition} Recall, the standard Armijo condition accepts a proposed iterate, $\psi$, from a current iterate, $\theta$, if
$F(\psi) \leq F(\theta) + \rho \dot F(\theta)^\intercal(\psi - \theta)$, where $\rho > 0$ is (typically) a small relaxation parameter \cite[see][Eq. 1.11]{bertsekas1999}. In \cite{grippo1986nonmonotone,grippo1989truncated,grippo1991class}, the standard Armijo condition is generalized such that given $w \in \mathbb{N}$ of the most recent iterates, $\lbrace \theta_{k},\ldots,\theta_{k-w+1} \rbrace \subset \mathbb{R}^n$, a proposed iterate $\psi \in \mathbb{R}^n$ is accepted if
$F(\psi) \leq \max \lbrace F(\theta_j) : j=k-w+1,\ldots, k \rbrace + \rho \dot F(\theta_k)^\intercal (\psi - \theta_k)$, where $\rho > 0$ again plays the role of a relaxation parameter. This generalized Armijo condition, called a nonmonotone Armijo condition, allows for a nonmonotonic change in the objective function between iterates when $w > 1$, and reduces to the standard Armijo condition when $w = 1$. 

To compare our nonmonotone Armijo condition to the above versions, we need to define a subsequence of $\mathbb{N}$ for when our outer loop iterates are distinct.
Let
\begin{equation} \label{eqn-iterate-subsequence}
\ell_0 = 0\quad\mathrm{and}\quad \ell_t = \lbrace k > \ell_{t-1}: \theta_k \neq \theta_{\ell_{t-1}} \rbrace, ~ \forall t \in \mathbb{N},
\end{equation}
with the convention that $\ell_t = \infty$ if no finite $k$ can be found to satisfy the property, and $\ell_{t} = \infty$ if $\ell_{t-1} = \infty$. Also, let $L: \mathbb{N} \cup \lbrace 0 \rbrace \to \mathbb{N} \cup \lbrace 0 \rbrace$ such that
\begin{equation} \label{eqn-iterate-subsequence-map}
L(k) = \max \lbrace t : \ell_t \leq k \rbrace,
\end{equation}
which specifies the element of $\lbrace \ell_t \rbrace$ that produced the most recent distinct iterate up to iterate $k$.
With this notation, define
\begin{equation} \label{eqn-nonmonotone-threshold}
\tau_{\obj}^k = \max\left\lbrace F(\theta_{\ell_{\max\lbrace L(k) - w + 1, 0 \rbrace}}),F(\theta_{\ell_{\max\lbrace L(k) - w + 1, 0 \rbrace+1}}),\ldots ,F(\theta_{\ell_{L(k)}}) \right\rbrace,
\end{equation}
which sets $\tau_{\obj}^k$ to the objective of one of the $w$ most recent, distinct outer loop iterates.
To develop a familiarity with these quantities, some properties are collected in the following lemma, and a toy example of their behavior is shown in \cref{figure-objective-threshold-diagram}. 

\begin{lemma} \label{result-simple-properties-accepted-iterates}
Let $\lbrace \ell_t : t+1 \in \mathbb{N}\rbrace$, $L: \mathbb{N} \cup \lbrace 0 \rbrace \to \mathbb{N} \cup \lbrace 0 \rbrace$, and $\lbrace \tau_{\obj}^k : k+1 \in \mathbb{N} \rbrace$ be defined as in \cref{eqn-iterate-subsequence,eqn-iterate-subsequence-map,eqn-nonmonotone-threshold}, respectively. Then, the following properties hold. 
\begin{remunerate}
\item For any $t+1 \in \mathbb{N}$, if $\ell_t, \ell_{t+1} < \infty$, then $\theta_{\ell_t} = \theta_{\ell_t + 1} = \cdots = \theta_{\ell_{t+1} - 1}$. Hence, $\forall k +1 \in \mathbb{N}$, $\theta_{k} = \theta_{\ell_{L(k)}}$. 
\item For any $t+1 \in \mathbb{N}$, if $\ell_t, \ell_{t+1} < \infty$, then $\tau_{\obj}^{\ell_t} = \tau_{\obj}^{\ell_t + 1} = \cdots = \tau_{\obj}^{\ell_{t+1}-1}$. Hence, $\forall k +1 \in \mathbb{N}$, $\tau_{\obj}^{k} = \tau_{\obj}^{\ell_{L(k)}}$.
\end{remunerate}
\end{lemma}

\begin{figure}[H]
\centering
\small
\input{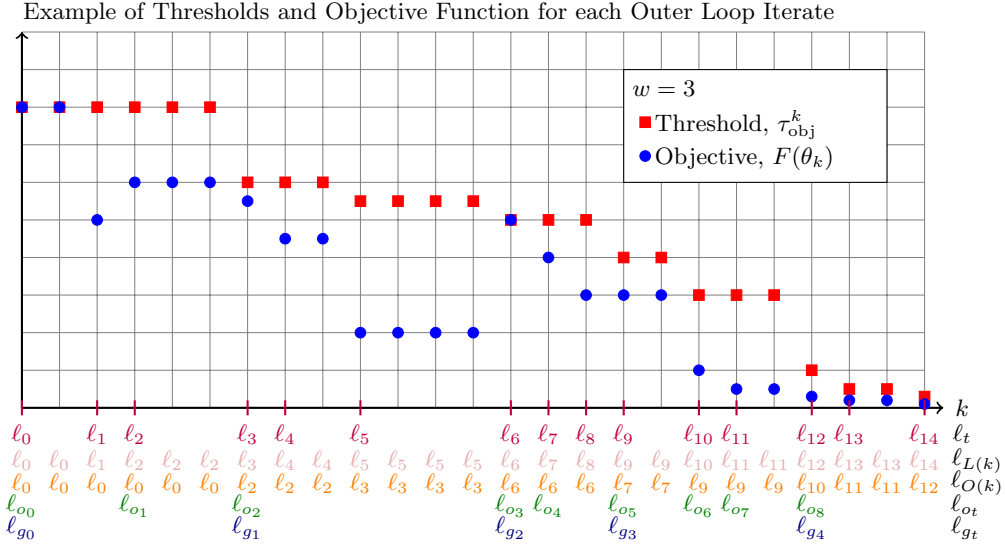}
\caption{A toy example to show the possible behaviors of the objective function, threshold, and different iteration counters. The horizontal axis is the outer loop iteration counter, starting at zero and running to twenty three with each vertical gray line indicating a new outer loop iteration. The vertical axis corresponds to the objective function and threshold values. The values of $\lbrace \ell_t \rbrace$, $\lbrace \ell_{L(k)} \rbrace$, $\lbrace \ell_{O(k)} \rbrace$, $\lbrace \ell_{o_t} \rbrace$, and $\lbrace \ell_{g_t} \rbrace$ are indicated below the horizontal axis.\label{figure-objective-threshold-diagram}}
\end{figure}

With this notation, we return to comparing our nonmonotone Armijo condition against the above two versions. First, unlike the two aforementioned conditions, when the trigger iterate $j$ exceeds $1$, our condition for acceptance, $F(\psi_{j}^k) < \tau_{\obj}^k + \rho \delta_k \alpha_0^k \dot F(\theta_k)^\intercal \gamma_0^k$, has a fixed right hand side that does not depend on $\psi_j^k - \theta_k$. Second, when $j>1$ and $w=1$, then our condition does not reduce to the standard Armijo condition as is the case for the condition in \cite{grippo1986nonmonotone,grippo1989truncated,grippo1991class}. When $j = 1$, then our condition reduces to that of \cite{grippo1986nonmonotone,grippo1989truncated,grippo1991class}. When $j = 1$ and $w=1$, then our condition reduces to the standard Armijo condition.

\paragraph{Test Parameter Selection} In \Cref{alg-general-algorithm}, the test parameter selection procedure should be designed to satisfy the following properties. We begin with requirements for $\lbrace \tau_{\iter,\exit}^k : k+1 \in \mathbb{N} \rbrace$.
\begin{property} \label{property-radius-nonnegative}
For all $k+1 \in \mathbb{N}$, $\tau_{\iter,\exit}^k \geq 0$.
\end{property}
\begin{property} \label{property-radius-bounded}
For any $\theta \in \mathbb{R}^n$, $\exists \overline \tau_{\exit}(\theta) \geq 0$ such that if $\theta = \theta_k$ then $\tau_{\iter,\exit}^k \leq \overline{\tau}_{\exit}(\theta)$.
\end{property}
We now consider requirements on $\tau_{\gra,\low}^k$ and $\tau_{\gra,\upp}^k$.
\begin{property} \label{property-grad-lower}
For any $\theta \in \mathbb{R}^n$, for all $k+1 \in \mathbb{N}$ such that if $\theta_k = \theta$, then $\tau_{\gra,\low}^k \in (0, \inlinenorm{ \dot F(\theta)}_2)$.
\end{property}
\begin{property} \label{property-grad-upper}
For any $\theta \in \mathbb{R}^n$, for all $k+1 \in \mathbb{N}$ such that if $\theta_k = \theta$, then $\tau_{\gra,\upp}^k > \inlinenorm{ \dot F(\theta)}_2$.
\end{property}
Finally, we consider requirements on $\tau_{\iter,\max}^k$.
\begin{property} \label{property-iter-max}
There exists an $\overline \tau_{\max}$ such that $1 \leq \tau_{\iter,\max}^k \leq \overline \tau_{\max}$ for all $k + 1 \in \mathbb{N}$.
\end{property}
The above properties prevent the inner loop from being triggered at $j = 0$ since $\inlinenorm{ \psi_0^k - \theta_k } = 0 \not > \tau_{\iter,\exit}^k \geq 0$ by \cref{property-radius-nonnegative}; $\inlinenorm{\dot F(\psi_0^k)}_2 = \inlinenorm{\dot F(\theta_k)}_2 \in (\tau_{\gra,\low}^k, \tau_{\gra,\upp}^k)$ by \cref{property-grad-lower,property-grad-upper}; and $\tau_{\iter,\max}^k \geq 1$ by \cref{property-iter-max}.

\subsection{Global Convergence Analysis} \label{subsec-global-convergence}

We now establish the asymptotic properties of our general methodology from a global convergence perspective (\textbf{by ignoring the outer loop stopping condition in \cref{alg-general-algorithm}}). We follow the following general steps in our analysis.
\begin{remunerate}
\item \textit{Accepting new iterates.} We show that the procedure must accept a new iterate unless it has already found a first-order stationary point. Mathematically, we prove $\ell_{t+1} < \infty$ if $\dot F(\theta_{\ell_t}) \neq 0$.
\item \textit{Objective function remains bounded.} We analyze the behavior of the thresholds, $\lbrace \tau_{\obj}^k : k+1 \in \mathbb{N} \rbrace$, to show that the objective function cannot diverge.
\item \textit{Analysis of a gradient subsequence.} We analyze the asymptotic behavior of the gradients at the iterates to conclude that either a stationary point is found in finite time, or a subsequence of the gradients tends to zero if certain growth conditions hold on the constants arising in \cref{as-loc-lip-cont,property-stepsize-lower,property-stepsize-upper,property-stepdirection-negative,property-stepdirection-propgrad,property-radius-bounded}.
\end{remunerate}

We underscore several points. First, under our rather general problem setting, we cannot produce a useful complexity analysis. Second, under our general properties, we also do not provide a local convergence analysis as this would be specialized to the design choices of the end-user. We will leave such specialization to the future and focus on a general theory here.

\paragraph{Accepting new iterates} For each $k+1 \in \mathbb{N}$, let $j_k \in \mathbb{N}$ denote the triggering iterate for the inner loop, which is bounded by $\overline{\tau}_{\max}$ if \cref{property-iter-max} holds and is nonzero if \cref{property-radius-nonnegative,property-grad-lower,property-grad-upper,property-iter-max} hold. To begin, we show that inner loop iterates remain in a compact set until an accepted iterate is found.

\begin{lemma} \label{result-iterates-in-compactsets}
Suppose \cref{problem} satisfies \cref{as-loc-lip-cont}, and is solved using \cref{alg-general-algorithm} with  
\cref{property-iter-max,property-stepsize-upper,property-stepdirection-propgrad,property-radius-bounded,property-radius-nonnegative,property-grad-upper}
initialized at any $\theta_0 \in \mathbb{R}^n$.
Then, for every $k+1 \in \mathbb{N}$, there exists a compact $C_k \subset \mathbb{R}^n$ such that $\theta_k \in C_k$, $\lbrace \psi_{1}^k,\ldots, \psi_{j_k}^k \rbrace \subset C_k$, and, if $\theta_{k} = \theta_{k+1}$, then $C_{k} \supset C_{k+1}$.
\end{lemma}
\begin{proof}
For any $\theta \in \mathbb{R}^n$, let $\overline \tau_{\exit}(\theta)$ be as in \cref{property-radius-bounded}. 
For any $\theta \in \mathbb{R}^n$, define $\mathcal{B}(\theta) = \lbrace \psi: \inlinenorm{ \psi - \theta}_2 \leq \overline \tau_{\exit}(\theta) \rbrace$; and define $\mathcal G(\theta) = \sup_{\psi \in \mathcal{B}(\theta)} \inlinenorm{ \dot F(\psi)}_2$, which is finite because of the continuity of the gradient function under \cref{as-loc-lip-cont}. 
Using \cref{property-stepsize-upper,property-stepdirection-propgrad,property-radius-bounded}, define
\begin{equation} \label{eqn-compactset-outerloop}
C_k = \lbrace \psi : \inlinenorm{ \psi - \theta_k }_2 \leq \overline \tau_{\exit}(\theta_k) + \delta_k \overline\alpha (\mathcal{B}(\theta_k)) \overline g( \mathcal{B}(\theta_k)) \mathcal G(\theta_k) \rbrace, ~\forall k+1 \in \mathbb{N}.
\end{equation}
We first show $C_k$ satisfies the desired properties. First, $\theta_k \in C_k$ as the radius of the ball defining $C_k$ is non-negative. Second, by definition, the inner loop iterate $\psi_{i}^k$ satisfies $\inlinenorm{ \psi_{i}^k - \theta_k}_2 \leq \tau_{\iter,\exit}^k \leq \overline \tau_{\exit}(\theta_k)$ for $i=1,\ldots,j_k-1$. In other words, $\psi_{i}^k \in \mathcal{B}(\theta_k)$ for $i=1,\ldots,j_k-1$. 

As a result, by \cref{property-stepsize-upper}, $\alpha_{j_k-1}^k \leq \overline \alpha (\mathcal{B}(\theta_k))$; and, by \cref{property-grad-upper,as-loc-lip-cont}, $\inlinenorm{ \gamma_{j_k-1}^k}_2 \leq \overline g(\mathcal{B}(\theta_k)) \inlinenorm{ \dot F(\psi_{j_k-1}^k)}_2 \leq \overline g(\mathcal{B}(\theta_k))\mathcal G(\theta_k)$.
Putting these pieces together,
\begin{align}
\inlinenorm{ \psi_{j_k}^k - \theta_k}_2 
&\leq \inlinenorm{ \psi_{j_k}^k - \psi_{j_k-1}^k}_2 + \inlinenorm{\psi_{j_k -1}^k - \theta_k}_2 \\
&\leq \inlinenorm{ \delta_k \alpha_{j_k-1}^k  \gamma_{j_k-1}^k }_2 + \overline \tau_{\exit}(\theta_k) \\
&\leq \delta_k \overline \alpha (\mathcal{B}(\theta_k)) \overline g (\mathcal{B}(\theta_k)) \mathcal G(\theta_k) + \overline \tau_{\exit}(\theta_k).
\end{align}
To summarize, $\psi_{j_k}^k \in C_k$ and $\psi_{i}^k \in \mathcal{B}(\theta_k) \subset C_k$ for $i=1,\ldots,j_k-1$. 

Finally, when $\theta_{k} = \theta_{k+1}$, $\delta_{k+1} = \sigma_{\low} \delta_k < \delta_k$ since $\sigma_{\low} \in (0,1)$ as required by \cref{alg-general-algorithm}. Plugging this information in \cref{eqn-compactset-outerloop}, it follows that $C_{k} \supset C_{k+1}$.
\end{proof}

With the existence of $\lbrace C_k : k+1 \in \mathbb{N} \rbrace$ established, we now show, there is a sufficiently small choice of $\delta_k$ such that the triggering iterate of the inner loop, $\psi_{j_k}^k$, will satisfy our nonmonontone Armijo condition \textit{despite} the condition not depending on the difference between the terminal iterate and initial iterate.

\begin{lemma} \label{result-sufficient-scaling}
Suppose \cref{problem} satisfies \cref{as-loc-lip-cont}, and is solved using \cref{alg-general-algorithm} with  
\cref{property-iter-max,property-stepsize-lower,property-stepsize-upper,property-stepdirection-negative,property-stepdirection-propgrad,property-radius-bounded,property-grad-lower,property-radius-nonnegative,property-grad-upper}
initialized at any $\theta_0 \in \mathbb{R}^n$.
Let $C \subset \mathbb{R}^n$ be a compact set such that $C_k \subset C$ for some $k+1 \in \mathbb{N}$. Let $\mathcal L(C)$ denote the Lipschitz rank of the gradient function over $C$. If $\dot F(\theta_k) \neq 0$ and
\begin{equation}
\delta_k < \frac{2 (1 - \rho) \underline g(C) }{\overline g(C)^2 \mathcal L(C) \overline \alpha(C)},
\end{equation}
then $F(\psi_{j_k}^k) < F(\theta_k) + \rho \delta_k \alpha_{0}^k \dot F(\theta_k)^\intercal \gamma_0^k \leq \tau_{\obj}^k + \rho \delta_k \alpha_{0}^k \dot F(\theta_k)^\intercal \gamma_0^k$. 
\end{lemma}

\begin{proof}
We recall several facts. By \cref{result-iterates-in-compactsets}, $\lbrace \psi_{j}^k : j=0,\ldots, j_k \rbrace \subset C_k \subset C$. Hence, by \cref{property-stepsize-lower,property-stepsize-upper,property-radius-nonnegative}, $0 < \underline\alpha(C) \leq \alpha_j^k \leq \overline\alpha(C) < \infty$ for $j = 0,\ldots j_k-1$. Moreover, by \cref{property-stepdirection-negative,property-stepdirection-propgrad}, $-\underline g(C)\inlinenorm{ \dot F(\psi_j^k)}_2^2 \geq \dot F(\psi_j^k)^\intercal \gamma_j^k$ and $\overline g(C) \inlinenorm{ \dot F(\psi_j^k)}_2 \geq \inlinenorm{ \gamma_j^k}_2$ for $j=0,\ldots,j_k-1$. Note, $j_k \neq 0$ under the given properties as described at the end of \cref{subsec-method}.

Now, we use these facts, to convert the hypothesis on $\delta_k$ into a more useful form. Specifically, by \cref{property-stepsize-upper}, $\delta_k \alpha_j^k \leq \delta_k \overline\alpha(C) < [2(1-\rho) \underline g(C)]/[ \overline g^2 (C) \mathcal L(C)]$ for $j=0,\ldots,j_k -1$. 
By \cref{alg-general-algorithm} and \cref{property-stepsize-lower}, $\delta_k \alpha_j^k > 0$, which implies $(\delta_k \alpha_j^k)^2 \overline g(C)^2 \mathcal L(C)/2 < (\delta_k \alpha_j^k)(1-\rho) \underline g(C)$ (note, we can take $\mathcal L(C) > 0$). 

By hypothesis and \cref{property-grad-lower}, 
$\inlinenorm{\dot F(\psi_j^k)}_2^2 > 0$ for $j=0,\ldots,j_k-1$. 
So, $(\delta_k \alpha_j^k)^2 \overline g(C)^2 \inlinenorm{\dot F(\psi_j^k)}_2^2 \mathcal L(C)/2 < (\delta_k \alpha_j^k) (1-\rho) \underline g(C) \inlinenorm{\dot F(\psi_j^k)}_2^2$. Using \cref{property-stepdirection-negative,property-stepdirection-propgrad},
$\inlinenorm{ \delta_k \alpha_j^k \gamma_j^k}_2^2 \mathcal L(C)/2  < -(\delta_k \alpha_j^k)(1-\rho) \dot F(\psi_j^k)^\intercal \gamma_j^k$. In other words, for $j=0,\ldots,j_k-1$,
\begin{equation} \label{eqn-objective-reduction}
(\delta_k \alpha_j^k) (1-\rho) \dot F(\psi_j^k)^\intercal \gamma_j^k + \frac{\mathcal L(C)}{2} \inlinenorm{ \delta_k \alpha_j^k \gamma_j^k}_2^2 < 0.
\end{equation}
In fact, since $(1-\rho) \in (0,1)$ and $\dot F(\psi_j^k)^\intercal \gamma_j^k < 0$ by \cref{property-stepdirection-negative}, the $(1-\rho)$ can be replaced with $1$ and the inequality will hold.

We use this relationship \cref{eqn-objective-reduction} in two ways. First, by Taylor's theorem, \cref{as-loc-lip-cont} and \cref{eqn-objective-reduction}, for $j=0,\ldots,j_k-1$,
\begin{equation}
F(\psi_{j+1}^k) \leq F(\psi_j^k) + \alpha_j^k \delta_k \dot F(\psi_j^k)^\intercal \gamma_j^k + \frac{\mathcal L(C)}{2} \inlinenorm{ \delta_k \alpha_j^k \gamma_j^k }_2^2 < F(\psi_j^k).
\end{equation}

In particular, at $j=0$, by \cref{eqn-objective-reduction},
\begin{align}
F(\psi_1^k) 
&\leq F(\psi_0^k) + (1-\rho)\alpha_0^k \delta_k \dot F(\psi_0^k)^\intercal \gamma_0^k + \frac{\mathcal L(C)}{2} \inlinenorm{ \delta_k \alpha_0^k \gamma_0^k }_2^2  + \rho \alpha_0^k \delta_k \dot F(\psi_0^k)^\intercal \gamma_0^k \\
&< F(\psi_0^k) + \rho \alpha_0^k \delta_k \dot F(\psi_0^k)^\intercal \gamma_0^k.
\end{align}
Putting this together with $\psi_0^k = \theta_k$, $F(\psi_{j_k}^k) \leq F(\psi_1^k) < F(\theta_k) + \rho \alpha_0^k \delta_k \dot F(\theta_k)^\intercal \gamma_0^k$. 

Finally, $\theta_k = \theta_{\ell_{L(k)}}$ and so $F(\theta_k) = F(\theta_{\ell_{L(k)}}) \leq \tau_{\obj}^k$ by \cref{eqn-nonmonotone-threshold}.
\end{proof}

We now combine these two facts to show that the procedure must always accept a new iterate as long as a stationary point has yet to be found.

\begin{theorem} \label{result-accept-new-iterates}
Suppose \cref{problem} satisfies \cref{as-loc-lip-cont}, and is solved using \cref{alg-general-algorithm} with  
\cref{property-iter-max,property-stepsize-lower,property-stepsize-upper,property-stepdirection-negative,property-stepdirection-propgrad,property-radius-bounded,property-grad-lower,property-radius-nonnegative,property-grad-upper}
initialized at any $\theta_0 \in \mathbb{R}^n$. 
Let $\lbrace \ell_t : t+1 \in \mathbb{N} \rbrace$ be defined as in \cref{eqn-iterate-subsequence}. Then, for any $t+1\in \mathbb{N}$, if $\ell_t < \infty$ and $\dot F(\theta_{\ell_t}) \neq 0$, then $\ell_{t+1} < \infty$.
\end{theorem}
\begin{proof}
The proof is by induction. As the proof of the base case (i.e., $\ell_1 < \infty$) uses the same argument as the conclusion (i.e., if $\ell_t < \infty$ then $\ell_{t+1} < \infty$), we show the conclusion. 
To this end, suppose $\lbrace \ell_0,\ldots,\ell_{t} \rbrace$ are finite and $\dot F(\theta_{\ell_t}) \neq 0$. For a contradiction, suppose $\theta_k = \theta_{\ell_t}$ for all $k \geq \ell_t$. Then, by \cref{result-iterates-in-compactsets}, the compact sets $C_{k} \subset C_{\ell_t}$ for all $k \geq \ell_t$. By \cref{alg-general-algorithm} and $\sigma_{\low} \in (0,1)$, then $\delta_k = \sigma_{\low}^{k-\ell_t} \delta_{\ell_t}$. There exists a $k \geq \ell_t$ such that
\begin{equation}
\delta_k = \sigma_{\low}^{k - \ell_t} \delta_{\ell_t} < \frac{2 (1 - \rho) \underline g (C_{\ell_t}) }{ \overline g(C_{\ell_t})^2 \mathcal L(C_{\ell_t}) \overline \alpha(C_{\ell_t})},
\end{equation}
which, by \cref{result-sufficient-scaling} and using $\dot F(\theta_k) = \dot F(\theta_{\ell_t}) \neq 0$, implies $\psi_{j_k}^k$ satisfies our nonmonotone Armijo condition. Hence, $\theta_{k+1} = \psi_{j_k}^k \neq \theta_{\ell_t}$. Therefore, $\ell_{t+1} = k+1 < \infty$.
\end{proof}

\paragraph{Objective function remains bounded} We begin by establishing properties of $\tau_{\obj}^k$. Recall, from \cref{eqn-nonmonotone-threshold}, the value of $\tau_{\obj}^k$ is determined by the maximum $F(\theta_{\ell_i})$ over $i\in\lbrace\max\lbrace L(k) - w + 1, 0 \rbrace,\ldots,L(k)\rbrace$. To keep track of which accepted iterate corresponds to $\tau_{\obj}^k$, a useful quantity to define is $O:\mathbb{N} \cup \lbrace 0 \rbrace \to \mathbb{N} \cup \lbrace 0 \rbrace$ such that
\begin{equation} \label{eqn-maxobj-index-func}
O(k) = \max \lbrace s \leq L(k) : F(\theta_{\ell_s}) = \tau_{\obj}^k \rbrace, ~k+1 \in \mathbb{N}.
\end{equation}
We collect simple facts about $O$ in the following lemma, and show an example of its behavior in \cref{figure-objective-threshold-diagram}.

\begin{lemma} \label{result-properties-maxobj-index-func}
Let $O:\mathbb{N} \cup \lbrace 0 \rbrace \to \mathbb{N} \cup \lbrace 0 \rbrace$ be defined as in \cref{eqn-maxobj-index-func}. If, for $k+1 \in \mathbb{N}$, $\dot F(\theta_k) \neq 0$, then the following properties hold.
\begin{remunerate}
\item $O(k) \in \lbrace \max\lbrace L(k) - w + 1,0 \rbrace,\ldots,L(k) \rbrace$;
\item $\tau_{\obj}^k = F(\theta_{\ell_{O(k)}})$;
\item If $O(k) \neq L(k)$, then $F(\theta_{\ell_i}) < \tau_{\obj}^k$ for $i \in \lbrace O(k)+1,\ldots,L(k) \rbrace$; and
\item $F(\theta_{\ell_i}) \leq \tau_{\obj}^k$ for $i \in \lbrace \max\lbrace L(k) - w + 1, 0 \rbrace, \ldots, O(k) \rbrace$. 
\end{remunerate}
\end{lemma}

\Cref{eqn-nonmonotone-threshold,result-properties-maxobj-index-func} indicate that $\tau_{\obj}^{k} > \tau_{\obj}^{k+1}$ only when $k+1$ is an accepted iterate and when $F(\theta_{\ell_{O(k)}})$ is no longer in the set defining $\tau_{\obj}^{k+1}$. In other words, $\tau_{\obj}^{k} > \tau_{\obj}^{k+1}$ when $k+1 = \ell_{L(k+1)}$ and $O(k) \not\in \lbrace \max\lbrace L(k+1) - w + 1, 0 \rbrace,\ldots,L(k+1) \rbrace$. We verify this rigorously now.

\begin{lemma} \label{result-threshold-decrease}
Suppose \cref{problem} satisfies \cref{as-loc-lip-cont}, and is solved using \cref{alg-general-algorithm} with  
\cref{property-iter-max,property-stepsize-lower,property-stepsize-upper,property-stepdirection-negative,property-stepdirection-propgrad,property-radius-bounded,property-grad-lower,property-radius-nonnegative,property-grad-upper}
initialized at any $\theta_0 \in \mathbb{R}^n$.
Let $\lbrace \ell_t : t+1 \in \mathbb{N} \rbrace$ be defined as in \cref{eqn-iterate-subsequence}, and let $O:\mathbb{N} \cup \lbrace 0 \rbrace \to \mathbb{N} \cup \lbrace 0 \rbrace$ be defined as in \cref{eqn-maxobj-index-func}. For any $k+1 \in \mathbb{N}$, if $\dot F(\theta_k) \neq 0$, then only one of the following holds.
\begin{remunerate}
\item $k+1 = \ell_{L(k+1)}$ (i.e., $\theta_{k+1} \neq \theta_k$), $O(k) = L(k) - w + 1$, and $\tau_{\obj}^k > \tau_{\obj}^{k+1}$;
\item $k+1 = \ell_{L(k+1)}$ (i.e., $\theta_{k+1} \neq \theta_k$), $O(k) \neq L(k) - w + 1$, $\tau_{\obj}^k = \tau_{\obj}^{k+1}$, and $O(k+1) = O(k)$; or
\item $\tau_{\obj}^k = \tau_{\obj}^{k+1}$.
\end{remunerate}
\end{lemma}
\begin{proof}
We recall several facts. First, $\theta_k = \theta_{\ell_{L(k)}}$. Second, either the triggering iterate is rejected, which is equivalent to $L(k+1) = L(k)$; or the triggering iterate is accepted, which is equivalent to $L(k+1) = L(k)+1$ and $\ell_{L(k+1)} = k+1$. 

Now, there are three cases to consider. First, $\theta_k = \theta_{k+1}$. Then, $\tau_{\obj}^k = \tau_{\obj}^{k+1}$ by \cref{eqn-nonmonotone-threshold}. Second, consider when $\theta_k \neq \theta_{k+1}$ and $O(k) > L(k) - w + 1$. Then, $O(k) \geq L(k+1) - w + 1$ since $L(k+1) = L(k)+1$. Using this fact with \cref{result-properties-maxobj-index-func}, 
$\tau_{\obj}^{k+1} = \max\lbrace F(\theta_{k+1}), F(\theta_{\ell_{O(k)}}) \rbrace$. 
Now, so long as $F(\theta_{k+1}) < F(\theta_{\ell_{O(k)}})$, then $\tau_{\obj}^{k+1} = \tau_{\obj}^{k}$ and $O(k+1) = O(k)$.
By our nonmonotone Armijo condition, since $\dot F(\theta_k) \neq 0$, $F(\theta_{\ell_{L(k+1)}}) = F(\theta_{k+1}) < \tau_{\obj}^k = F(\theta_{\ell_{O(k)}})$. Hence, $\tau_{\obj}^{k+1} = \tau_{\obj}^{k}$ and $O(k+1) = O(k)$.

Third, consider when $\theta_k \neq \theta_{k+1}$ and $O(k) = L(k) - w + 1$. Then, $O(k) = L(k) - w + 1 \geq 0$. By our facts about $O(k)$, $F(\theta_{\ell_i}) < \tau_{\obj}^k = F(\theta_{\ell_{O(k)}})$ for $i = L(k) - w + 2,\ldots, L(k)$. Using $L(k) + 1 = L(k+1)$ when $\theta_k \neq \theta_{k+1}$ and \cref{result-properties-maxobj-index-func}, $F(\theta_i) < \tau_{\obj}^k$ for $i = L(k+1) - w + 1,\ldots, L(k+1) - 1$. Moreover, by our nonmonotone Armijo condition and $\dot F(\theta_k) \neq 0$, $F(\theta_{\ell_{L(k+1)}}) = F(\theta_{k+1}) < \tau_{\obj}^k$. Hence, by the definition of $\tau_{\obj}^{k+1}$, $\tau_{\obj}^{k+1} < \tau_{\obj}^k$.
\end{proof}

Part of \Cref{result-threshold-decrease} is as follows: $\tau_{\obj}^{k+1} < \tau_{\obj}^{k}$ only when $k+1$ is an accepted iterate and $O(k) = L(k+1) - w$. In other words, the threshold only can decrease at $k+1$, if $k+1$ is an accepted iterate that is $w$ accepted iterates away from $O(k)$. 
This behavior motivates us to define the following sequence.
\begin{equation} \label{eqn-maxobj-index}
o_0 = 0\quad\text{and}\quad o_s = O(\ell_{o_{s-1} + w}),~\forall s \in \mathbb{N},
\end{equation}
with the convention of $o_s = \infty$ if $\ell_{o_{s-1} + w} = \infty$ (see \cref{figure-objective-threshold-diagram}). With this notation, we formalize the above discussion.

\begin{lemma} \label{result-maxobj-threshold-properties}
Suppose \cref{problem} satisfies \cref{as-loc-lip-cont}, and is solved using \cref{alg-general-algorithm} with  
\cref{property-iter-max,property-stepsize-lower,property-stepsize-upper,property-stepdirection-negative,property-stepdirection-propgrad,property-radius-bounded,property-grad-lower,property-radius-nonnegative,property-grad-upper}
at any $\theta_0 \in \mathbb{R}^n$. 
Let $\lbrace \ell_t : t+1 \in \mathbb{N} \rbrace$ be defined as in \cref{eqn-iterate-subsequence}, 
and let $\lbrace o_s : s + 1 \in \mathbb{N} \rbrace$ be defined as in \cref{eqn-maxobj-index} and let $o_{-1} = -w$. 
For any $s+1 \in \mathbb{N}$, if $o_s < \infty$ then one of the two holds.
\begin{remunerate}
\item $\exists \bar i \in \lbrace o_{s-1} + w - o_{s},\ldots,w \rbrace$ such that $\dot F(\theta_{\ell_{o_s + \bar i}}) = 0$, $\ell_{o_s+i} < \infty$ for all $i \in \lbrace o_{s-1}+w-o_s,\ldots,\bar i \rbrace$, and $\tau_{\obj}^k = F(\theta_{\ell_{o_s}})$ for all $k \in [\ell_{o_{s-1}+w},\max\lbrace\ell_{o_s+\bar{i}}-1, \ell_{o_{s-1}+w} \rbrace] \cap (\mathbb{N} \cup \lbrace 0 \rbrace)$; or
\item $\forall i \in \lbrace o_{s-1}+w-o_s,\ldots, w \rbrace$, $\ell_{o_s+i} < \infty$ and $\dot F(\theta_{\ell_{o_s + i}}) \neq 0$, and $\tau_{\obj}^k = F(\theta_{\ell_{o_s}})$ for all $k \in [\ell_{o_{s-1}+w},\ell_{o_{s}+w}-1] \cap (\mathbb{N} \cup \lbrace 0 \rbrace)$.
\end{remunerate}
\end{lemma}
\begin{proof}
We proceed by induction. At each step, we verify that $\ell_{o_s+i} < \infty$ and $\tau_{\obj}^{k} = F(\theta_{\ell_{o_s}})$ for all $k \in [ \ell_{o_{s-1} + w},\min\lbrace \ell_{o_s+i}-1, \ell_{o_{s-1} + w} \rbrace] \cap (\mathbb{N} \cup \lbrace 0 \rbrace)$. Then, there are two cases to deal with: either $\dot F(\theta_{\ell_{o_s+i}}) = 0$ or it does not. In the former case, we produce the first part of the result. In the latter case, we proceed with induction.

For the base case, $i=o_{s-1} + w - o_s$. Now, $o_s < \infty$ by hypothesis, which implies $\ell_{o_s + i} = \ell_{o_{s-1} + w} < \infty$. Moreover,
\begin{equation}
o_s = O(\ell_{o_{s-1}+w}) = \max\left \lbrace t \leq o_{s-1} + w : F(\theta_{\ell_t}) = \tau_{\obj}^{\ell_{o_{s-1} + w}} \right \rbrace,
\end{equation}
which requires $F(\theta_{\ell_{o_s}}) = \tau_{\obj}^{\ell_{o_{s-1} + w}}$. Hence, $\forall k \in [ \ell_{o_{s-1} + w}, \min \lbrace \ell_{o_s+i}-1, \ell_{o_{s-1} + w} \rbrace ] = \lbrace \ell_{o_{s-1} + w} \rbrace$, $F(\theta_{\ell_{o_s}}) = \tau_{\obj}^k$. 

Now, either $\dot F(\theta_{\ell_{o_{s-1} + w}}) = 0$ and $\bar i  = o_{s-1} + w - o_s$ or $\dot F(\theta_{\ell_{o_{s-1} + w}}) \neq 0$ and we can increment $i$.

For the hypothesis, we assume for some $i \in \lbrace o_{s-1} +w-o_s,\ldots, w - 1 \rbrace$, $\ell_{o_s+i} < \infty$, $\tau_{\obj}^{\ell_{o_s+i}} = F(\theta_{\ell_{o_s}})$, and $\dot F(\theta_{\ell_{o_s + i}}) \neq 0$. Furthermore, for any $\tilde i \in \lbrace o_{s-1}+w-o_s,\ldots, i \rbrace$, we assume $\tau_{\obj}^{\ell_{o_s+\tilde i}} = F(\theta_{\ell_{o_s}})$ and $\dot F(\theta_{\ell_{o_s + \tilde i}}) \neq 0$.

We now generalize the result to $i+1$. Since $\dot F(\theta_{\ell_{o_s+i}}) \neq 0$, $\ell_{o_s+i+1} < \infty$ by \cref{result-accept-new-iterates}. Moreover, by \cref{result-simple-properties-accepted-iterates} and the induction hypothesis, $\tau_{\obj}^k = \tau_{\obj}^{\ell_{o_s+i}} = F(\theta_{\ell_{o_s}})$ for $k \in \lbrace \ell_{o_s+i},\ldots, \ell_{o_s+i + 1} - 1 \rbrace$. There are now two cases to consider.

If $i = w- 1$, then either $\dot F(\theta_{\ell_{o_s + i + 1}})$ is zero or not (note, $o_s + i + 1 = o_s + w$). If it is zero, then we set $\bar i = w$ and the first part of the result is proven. If it is not zero, then the second part of the result is proven.

If $i < w - 1$, we must verify $\tau_{\obj}^{\ell_{o_s+i+1}} = F(\theta_{\ell_{o_s}})$. Suppose this is not true. By \cref{result-threshold-decrease}, $\tau_{\obj}^{\ell_{o_s+i+1}} < \tau_{\obj}^{\ell_{o_s + i + 1}-1} = F(\theta_{\ell_{o_s}})$ if and only if $O(\ell_{o_s + i + 1}-1) = L( \ell_{o_s+i+1} - 1) - w + 1 = o_s+i - w + 1$. When $i < w-1$, $O(\ell_{o_s+i+1}-1) = o_s + i - w + 1 < o_s$. Thus, $o_s \in \lbrace o_s + i - w + 2,\ldots, o_s + i + 1\rbrace$ and so
\begin{equation}
\tau_{\obj}^{\ell_{o_s+i+1}} = \max \lbrace F(\theta_{\ell_{o_s+i - w + 2}}),\ldots, F(\theta_{\ell_{o_s+i+1}}) \rbrace \geq F(\theta_{\ell_{o_s}}).
\end{equation}
This contradicts $\tau_{\obj}^{\ell_{o_s+i+1}} < F(\theta_{\ell_{o_s}})$. Hence, $\tau_{\obj}^{\ell_{o_s+i+1}} = F(\theta_{\ell_{o_s}})$.

Now, one of two options can now occur. First, $\dot F(\theta_{\ell_{o_s+i+1}}) = 0$ and $\bar i = o_s + i + 1$, which produces the first part of the result. Second, $\dot F(\theta_{\ell_{o_{s} + i + 1}}) \neq 0$, which concludes the induction proof.
\end{proof}

With the relevant information about the thresholds established, we can now conclude as follows about the behavior of the objective function.

\begin{theorem} \label{result-objective-function}
Suppose \cref{problem} satisfies \cref{as-loc-lip-cont}, and is solved using \cref{alg-general-algorithm} with  
\cref{property-iter-max,property-stepsize-lower,property-stepsize-upper,property-stepdirection-negative,property-stepdirection-propgrad,property-radius-bounded,property-grad-lower,property-radius-nonnegative,property-grad-upper}
at any $\theta_0 \in \mathbb{R}^n$. 
Let $\lbrace \ell_t : t+1 \in \mathbb{N} \rbrace$ be defined as in \cref{eqn-iterate-subsequence}, 
and let $\lbrace o_s : s + 1 \in \mathbb{N} \rbrace$ be defined as in \cref{eqn-maxobj-index}.
Then, one of the following occurs.
\begin{remunerate}
\item There exists a $t + 1 \in \mathbb{N}$ such that $\ell_t < \infty$ and $F(\theta_{\ell_t}) \leq F(\theta_0)$ and $\dot F(\theta_{\ell_t}) = 0$.
\item The elements of $\lbrace o_s : s+1 \in \mathbb{N} \rbrace$ are all finite; for any $s+1 \in \mathbb{N}$ and $\forall k \in [\ell_{o_s},\ell_{o_{s+1}}] \cap (\mathbb{N} \cup \lbrace 0 \rbrace)$, $F(\theta_k) \leq F(\theta_{\ell_{o_s}})$; and the sequence $\lbrace F(\theta_{\ell_{o_s}}) : s + 1 \in \mathbb{N} \rbrace$ is strictly decreasing.
\end{remunerate}
\end{theorem}
\begin{proof}
Let $o_{-1}=-w$. We proceed by induction on $s \in \mathbb{N} \cup \lbrace 0 \rbrace$.

For the base case, $s=0$, \cref{result-maxobj-threshold-properties} specifies two cases. The first case of \cref{result-maxobj-threshold-properties} supplies the first case of the present claim with $t \in \lbrace 0,\ldots,w\rbrace$. In the second case of \cref{result-maxobj-threshold-properties}, then two statements are true:
\begin{remunerate}
\item $\ell_w < \infty$;
\item $F(\theta_{\ell_{0}}) = \tau_{\obj}^k$ for all $k \in [0,\ell_{w}-1] \cap (\mathbb{N} \cup \lbrace 0 \rbrace)$.
\end{remunerate}
By the first statement, $o_1$ is finite. By both statements, our nonmonotone Armijo condition, and \cref{property-stepdirection-negative}, $F(\theta_{k+1}) \leq \tau_{\obj}^k = F(\theta_{\ell_0})$ for all $k \in [0,\ell_{w}-1] \cap (\mathbb{N} \cup \lbrace 0 \rbrace)$. In other words, 
$\forall k \in [\ell_{0},\ell_{w}] \cap (\mathbb{N} \cup \lbrace 0 \rbrace)$, $F(\theta_k) \leq F(\theta_{\ell_{0}})$. As $o_1 \leq w$, 
$\forall k \in [\ell_{0},\ell_{o_1}] \cap (\mathbb{N} \cup \lbrace 0 \rbrace)$, $F(\theta_k) \leq F(\theta_{\ell_0})$. Finally, by \cref{result-threshold-decrease,result-maxobj-threshold-properties}, $F(\theta_{\ell_0}) = \tau_{\obj}^{\ell_w - 1} > \tau_{\obj}^{\ell_w} = F(\theta_{\ell_{o_1}})$.

For the induction hypothesis, for some $s \in \mathbb{N} \cup \lbrace 0 \rbrace$, we assume the elements of $\lbrace o_t : t \in \lbrace 0,\ldots, s \rbrace \rbrace$ are finite; for all $t \in \lbrace 0,\ldots,\max\lbrace s-1, 0 \rbrace \rbrace$ and for all $k \in [\ell_{o_t},\ell_{o_{t+1}}] \cap (\mathbb{N} \cup \lbrace 0 \rbrace)$, $F(\theta_k) \leq F(\theta_{\ell_{o_t}})$; and for all $t \in \lbrace 0,\ldots,\max\lbrace s-1, 0 \rbrace \rbrace$, $F(\theta_{\ell_{o_t}}) > F(\theta_{\ell_{o_{t+1}}})$.

We now generalize to $s+1$. Since $o_s < \infty$ by the induction hypothesis, \cref{result-maxobj-threshold-properties} specifies two cases. In the first case of \cref{result-maxobj-threshold-properties}, there is a $\bar i \in \lbrace o_{s-1} + w - o_s,\ldots, w\rbrace$ such that
\begin{remunerate}
\item $\dot F(\theta_{\ell_{o_s + \bar i}}) = 0$,
\item $\ell_{o_s+i} < \infty$ for all $i \in \lbrace o_{s-1}+w-o_s,\ldots,\bar i \rbrace$, and
\item $\tau_{\obj}^k = F(\theta_{\ell_{o_s}})$ for all $k \in [\ell_{o_{s-1}+w},\min\lbrace \ell_{o_s+\bar i}-1, \ell_{o_{s-1}+w} \rbrace ] \cap (\mathbb{N} \cup \lbrace 0 \rbrace)$.
\end{remunerate} 
Let $ t = o_s + \bar i$. Then, by the first and second statements, $\ell_t < \infty$ and $\dot F(\theta_{\ell_t}) = 0$. By the third statement, our nonmonotone Armijo condition, and \cref{property-stepdirection-negative}, $F(\theta_{\ell_t}) = F(\theta_{\ell_{o_s+\bar i}}) < \tau_{\obj}^{\ell_{o_s+\bar i}-1} = F(\theta_{\ell_{o_s}})$. By the induction hypothesis, $F(\theta_{\ell_t}) \leq F(\theta_{\ell_{o_s}}) < F(\theta_{0})$. Hence, in the first case of \cref{result-maxobj-threshold-properties}, we conclude the first part of the current result.

In the second case of \cref{result-maxobj-threshold-properties}, we need to verify the three claims of the induction hypothesis for $s+1$. First, by \cref{result-maxobj-threshold-properties}, $\ell_{o_s + w} < \infty$, which implies $o_{s+1} < \infty$. Thus, the elements of $\lbrace o_t : t \in \lbrace 0,\ldots,s+1 \rbrace \rbrace$ are finite. 
Second, we must show $\forall k \in [\ell_{o_s}, \ell_{o_{s+1}}] \cap (\mathbb{N} \cup \lbrace 0 \rbrace)$, $F(\theta_k) \leq F(\theta_{\ell_{o_s}})$. By \cref{result-properties-maxobj-index-func} and the definition of $o_s$, $F(\theta_{\ell_t}) < F(\theta_{\ell_{o_s}})$ for all $t \in \lbrace \min\lbrace o_s+1, o_{s-1}+w \rbrace,\ldots, o_{s-1}+w \rbrace$. Therefore, $F(\theta_k) \leq F(\theta_{\ell_{o_s}})$ for all $k \in [\ell_{o_s},\ell_{o_{s-1}+w}] \cap (\mathbb{N} \cup \lbrace 0 \rbrace)$. By the second part of \cref{result-maxobj-threshold-properties}, $F(\theta_{k}) \leq \tau_{\obj}^{k-1} = F(\theta_{\ell_{o_s}})$ for all $k \in [\ell_{o_{s-1}+w}+1,\ell_{o_s+w}] \cap ( \mathbb{N} \cup \lbrace 0 \rbrace )$. Hence, by the induction hypothesis, for all $t \in \lbrace 0,\ldots,\max\lbrace s, 0 \rbrace \rbrace$ and for all $k \in [\ell_{o_t},\ell_{o_{t+1}}] \cap (\mathbb{N} \cup \lbrace 0 \rbrace)$, $F(\theta_k) \leq F(\theta_{\ell_{o_t}})$. Finally, by \cref{result-threshold-decrease,result-maxobj-threshold-properties}, $F(\theta_{\ell_{o_s}}) = \tau_{\obj}^{\ell_{o_s+w} - 1} > \tau_{\obj}^{\ell_{o_s+w}} = F(\theta_{\ell_{o_{s+1}}})$. This concludes the proof by induction.
\end{proof}

\paragraph{Analysis of a gradient subsequence} We study a specific subsequence of the accepted iterates to show that the gradient function evaluated along this subsequence \textit{can} be well-behaved based on the constants in \cref{property-iter-max,property-stepsize-lower,property-stepsize-upper,property-stepdirection-negative,property-stepdirection-propgrad,property-radius-bounded,property-grad-lower,property-radius-nonnegative,property-grad-upper} and the local properties of the Lipschitz rank, $\mathcal{L}(\cdot)$. To specify this sequence, letting $\lbrace \ell_t : t+1 \in \mathbb{N} \rbrace$ and $\lbrace o_s : s + 1 \in \mathbb{N} \rbrace$
be defined as in \cref{eqn-iterate-subsequence,eqn-maxobj-index} (respectively), define
\begin{equation} \label{eqn-gradient-index}
g_0 = 0 \quad\text{and}\quad g_u = \min\lbrace o_s \geq g_{u-1} + w \rbrace, ~\forall u \in \mathbb{N},
\end{equation}
with the convention $g_u = \infty$ if $g_{u-1} = \infty$ or if no finite $o_s$ can be found (see \cref{figure-objective-threshold-diagram}). With this notation, we have the following result, which we emphasize does not depend on the scaling constants $\lbrace \delta_k \rbrace$ and only on the user-designed constants in \cref{property-iter-max,property-stepsize-lower,property-stepsize-upper,property-stepdirection-negative,property-stepdirection-propgrad,property-radius-bounded,property-grad-lower,property-radius-nonnegative,property-grad-upper} and the local properties of the Lipschitz rank.

\begin{lemma} \label{result-zoutendjik}
Suppose \cref{problem} satisfies \cref{as-bounded below,as-loc-lip-cont}, and is solved using \cref{alg-general-algorithm} with  
\cref{property-iter-max,property-stepsize-lower,property-stepsize-upper,property-stepdirection-negative,property-stepdirection-propgrad,property-radius-bounded,property-grad-lower,property-radius-nonnegative,property-grad-upper}
at any $\theta_0 \in \mathbb{R}^n$. 
Let $\lbrace \ell_t : t+1 \in \mathbb{N} \rbrace$ be defined as in \cref{eqn-iterate-subsequence}, 
let $\lbrace o_s : s + 1 \in \mathbb{N} \rbrace$ be defined as in \cref{eqn-maxobj-index},
and let $\lbrace g_u : u+1 \in \mathbb{N} \rbrace$ be defined as in \cref{eqn-gradient-index}.
Let $\lbrace C_k : k+1 \in \mathbb{N} \rbrace$ be a sequence of compact sets in $\mathbb{R}^n$ satisfying: $\theta_{k} \in C_k$; $\lbrace \psi_{1}^k,\ldots,\psi_{j_k}^k \rbrace \subset C_k$; and if $\theta_{k+1} = \theta_k$ then $C_{k+1} \subset C_{k}$ (see \cref{result-iterates-in-compactsets}).
Then, one of the following occurs.
\begin{remunerate}
\item There exists a $t + 1 \in \mathbb{N}$ such that $\ell_t < \infty$ and $F(\theta_{\ell_t}) \leq F(\theta_0)$ and $\dot F(\theta_{\ell_t}) = 0$.
\item The elements of $\lbrace g_u: u+1 \in \mathbb{N} \rbrace$ are all finite, and
\begin{equation}
\sum_{u=1}^\infty 
\underline\alpha (C_{\ell_{g_u-1}})
\underline{g}(C_{\ell_{g_u-1}})
\inlinenorm{\dot F(\theta_{\ell_{g_u-1}})}_2^2 
< \infty.
\end{equation}
\item The elements of $\lbrace g_u : u+1 \in\mathbb{N} \rbrace$ are all finite, $\cup_{k=0}^\infty C_k$ is unbounded, and there exists a subsequence $\mathcal U \subseteq \mathbb{N}$ such that
\begin{equation}
  \sum_{u\in\mathcal U}
  \frac{\underline g(C_u')^2}{\overline g(C_u')^2}
  \frac{\underline \alpha(C_u')}{\overline \alpha(C_u')}
  \frac{\inlinenorm{\dot F(\theta_{\ell_{g_u-1}})}_2^2}{\mathcal L(C_u')} < \infty,
  ~\text{where}~C_u' = \cup_{k = \ell_{g_{u-1}}}^{\ell_{g_u}-1} C_k.
\end{equation}
\end{remunerate}
\end{lemma}
\begin{proof}
  By \cref{result-objective-function}, either we are in the first case of the result or $\lbrace o_s : s+1 \in \mathbb{N} \rbrace$ are all finite. Thus, when all elements of $\lbrace o_s : s+1 \in \mathbb{N} \rbrace$ are finite, then the elements of $\lbrace g_u : u+1 \in \mathbb{N} \rbrace$ are all finite. We divide this situation into two cases, which correspond to the second and third parts of the result. 

  First, we consider the case that $\lbrace \delta_k : k+1 \in \mathbb{N} \rbrace$ are bounded from below by $\underline \delta > 0$. We now use this fact with properties of $\lbrace g_u : u+1 \in \mathbb{N} \rbrace$ and the algorithm to conclude the second case of the result. Recall, by our nonmonotone Armijo condition,
  \begin{equation}
    F(\theta_{\ell_{g_{u}}}) < \tau_{\obj}^{\ell_{g_{u}}-1} + \rho \delta_{\ell_{g_{u}}-1} \alpha_0^{\ell_{g_{u}}-1} \dot F(\theta_{\ell_{g_u}-1})^\intercal \gamma_0^{\ell_{g_{u}}-1}, ~\forall u \in \mathbb{N}.
  \end{equation}
  By construction, $\theta_{\ell_{g_u}-1} = \theta_{\ell_{g_u-1}}$. Moreover, $C_{\ell_{g_u}-1} \subseteq C_{\ell_{g_u-1}}$. Applying this, \cref{property-stepdirection-negative,property-stepsize-lower}, the lower bound on $\lbrace \delta_k : k+1 \in \mathbb{N} \rbrace$, and rearranging, we obtain
  \begin{equation}
    \underline \alpha (C_{\ell_{g_u-1}}) \underline g (C_{\ell_{g_u-1}}) \inlinenorm{ \dot F(\theta_{\ell_{g_u-1}})}_2^2 < \frac{\tau_{\obj}^{\ell_{g_{u}}-1} - F(\theta_{\ell_{g_u}})}{\rho \underline\delta}
  \end{equation}
  Now, $\tau_{\obj}^{\ell_{g_u}-1} = F(\theta_{\ell_{o_s}})$ where $o_s \in \lbrace g_u - w,\ldots,g_u-1 \rbrace$. Since $g_{u-1} \leq g_{u} - w$ by construction, the second part of \cref{result-objective-function} states $F(\theta_{\ell_{o_s}}) \leq F(\theta_{\ell_{g_{u-1}}})$.
  Hence,
  \begin{equation}
    \underline \alpha (C_{\ell_{g_u-1}}) \underline g (C_{\ell_{g_u-1}}) \inlinenorm{ \dot F(\theta_{\ell_{g_u-1}})}_2^2 < \frac{F(\theta_{\ell_{g_{u-1}}}) - F(\theta_{\ell_{g_u}})}{\rho \underline\delta}
  \end{equation}
  Taking the sum over all $u \in \mathbb{N}$ and using \cref{as-bounded below}, the right hand side is bounded by $[F(\theta_{0}) - F_{l.b.}]/[\rho\underline \delta]$, which is finite. The second case of the result follows.

  We now consider what happens when $\lbrace g_u : u+1 \in \mathbb{N} \rbrace$ are all finite and $\liminf_{k\to\infty} \delta_k = 0$. If $\liminf_{k\to\infty} \delta_k = 0$, there must be a subsequence of $\mathbb{N}$ of rejected outer loop iterates $\lbrace \theta_k \rbrace$. The existence of such a subsequence will be used twice: once to show that $\cup_{k=0}^\infty C_k$ is unbounded, and then to define a subsequence $\mathcal U$ of $\mathbb N$. 

  First, suppose $\cup_{k=0}^\infty C_k$ is bounded. Then, there exists a compact set $C$ such that $\cup_{k=0}^\infty C_k \subset C$. Let $k \in \mathbb N$ be the first time $\delta_k < 2(1-\rho)\sigma_{\low} \underline g (C)/\overline g (C)^2 \overline \alpha (C) \mathcal L (C)$. Then, $\psi_{j_{k-1}}^{k-1}$ failed the nonmonotone Armijo condition with $\delta_{k-1} < 2(1-\rho) \underline g (C)/\overline g (C)^2 \overline \alpha (C) \mathcal L (C)$. However, since $C_{k-1} \subseteq C$ and $\dot F(\theta_{k-1}) \not= 0$, this contradicts \cref{result-sufficient-scaling}. Hence, $\cup_{k=0}^\infty C_k$ is unbounded.

  Second, define $\mathcal U$ to be all $u \in \mathbb N$ such that $\lbrace \ell_{g_{u-1}}, ..., \ell_{g_u}-1 \rbrace$ contains a rejected iterate. For $u \in \mathcal U$, let
  $C_u' = \cup_{k = \ell_{g_{u-1}}}^{\ell_{g_u}-1} C_k$. Note, by the properties of $C_k$ and definition of $g_u$, there are at most $2w$ sets contributing to $C_u'$, implying that $C_u'$ is compact for all $u \in \mathcal U$. We now follow a similar set of steps as in the preceding case. By our nonmonotone Armijo condition,
  \begin{equation}
    F(\theta_{\ell_{g_{u}}}) < \tau_{\obj}^{\ell_{g_{u}}-1} + \rho \delta_{\ell_{g_{u}}-1} \alpha_0^{\ell_{g_{u}}-1} \dot F(\theta_{\ell_{g_u}-1})^\intercal \gamma_0^{\ell_{g_{u}}-1}, ~\forall u \in \mathcal U.
  \end{equation}
  By \cref{property-stepdirection-negative} and \cref{result-sufficient-scaling},
  \begin{equation}
    F(\theta_{\ell_{g_u}}) < \tau_{\obj}^{\ell_{g_{u}}-1} - \rho \delta_{\ell_{g_u}-1} \underline \alpha (C_{u}') \underline g (C_{u}') \inlinenorm{ \dot F(\theta_{\ell_{g_u-1}})}_2^2, ~\forall u \in \mathcal U.
  \end{equation}
  By \cref{result-sufficient-scaling} and construction of $\mathcal U$,
  \begin{equation}
    \delta_{\ell_{g_u}-1} \geq \frac{2(1-\rho)\underline g(C_u')\sigma_{\low}}{\overline\alpha(C_u')\overline g(C_u')^2\mathcal L(C_u')}.
  \end{equation}
  Putting these pieces together and rearranging,
  \begin{equation}
    \frac{\underline g(C_u')^2}{\overline g(C_u')^2}
    \frac{\underline \alpha(C_u')}{\overline \alpha(C_u')}
    \frac{\inlinenorm{\dot F(\theta_{\ell_{g_u-1}})}_2^2}{\mathcal L(C_u')} <
    \frac{\tau_{\obj}^{\ell_{g_u}-1} - F(\theta_{\ell_{g_u}})}{2\rho(1-\rho)\sigma_{\low}}, ~\forall u \in \mathcal U.
  \end{equation}
  Applying the second part of \cref{result-objective-function}, $\tau_{\obj}^{\ell_{g_u}-1} \leq F(\theta_{\ell_{g_{u-1}}})$, and $F(\theta_{\ell_{g_{u-1}}})-F(\theta_{\ell_{g_u}}) \geq 0, ~\forall u \in \mathbb N$. Hence,
  \begin{equation}
    \begin{aligned}
      \sum_{u\in\mathcal U}
      \frac{\underline g(C_u')^2}{\overline g(C_u')^2}
      \frac{\underline \alpha(C_u')}{\overline \alpha(C_u')}
      \frac{\inlinenorm{\dot F(\theta_{\ell_{g_u-1}})}_2^2}{\mathcal L(C_u')} &< \sum_{u\in\mathcal U} \frac{F(\theta_{\ell_{g_{u-1}}}) - F(\theta_{\ell_{g_u}})}{2\rho(1-\rho)\sigma_{\low}}\leq \sum_{u=1}^\infty \frac{F(\theta_{\ell_{g_{u-1}}}) - F(\theta_{\ell_{g_u}})}{2\rho(1-\rho)\sigma_{\low}}
    \end{aligned}
  \end{equation}
  The right-hand side is a telescoping sum which is finite by \cref{as-bounded below}.
\end{proof}

We now provide sufficient conditions which guarantee that a procedure within our methodology will find a region of the objective function whose gradient is nearly zero. In particular, we roughly say that the procedure terminates in finite time; or, if the iterates remain in a bounded region, then they will come arbitrarily close to a first order stationary point; or, if the Lipschitz rank grows at most quadratically (which improves the results of \cite{li2023convex}), then the iterates will find a region of the gradient function that is arbitrarily close to zero.

\begin{theorem} \label{result-gradient}
Suppose \cref{problem} satisfies \cref{as-bounded below,as-loc-lip-cont}, and is solved using \cref{alg-general-algorithm} with  
\cref{property-iter-max,property-stepsize-lower,property-stepsize-upper,property-stepdirection-negative,property-stepdirection-propgrad,property-radius-bounded,property-grad-lower,property-radius-nonnegative,property-grad-upper}
at any $\theta_0 \in \mathbb{R}^n$. 
Then, we have the following possible outcomes.
\begin{remunerate}
  \item (Finite Termination) There exists a $k + 1 \in \mathbb{N}$ such that $F(\theta_{k}) \leq F(\theta_0)$ and $\dot F(\theta_{k}) = 0$.
  \item (Infinite Iterates) The sequences $\lbrace \theta_k : k + 1 \in \mathbb N \rbrace$ is infinite, $\lbrace F(\theta_{k}) : k+1 \in \mathbb{N} \rbrace$ is bounded with a strictly decreasing subsequence. Let $\lbrace C_k : k+1 \in \mathbb{N} \rbrace$ be a sequence of compact sets in $\mathbb{R}^n$ satisfying: $\theta_{k} \in C_k$; $\lbrace \psi_{1}^k,\ldots,\psi_{j_k}^k \rbrace \subset C_k$; and if $\theta_{k+1} = \theta_k$ then $C_{k+1} \subset C_{k}$ (see \cref{result-iterates-in-compactsets}). Then, there are three outcomes.
  \begin{remunerate}
    \item If $\cup_{k\in\mathbb N} C_k$ is bounded, then $\lbrace \theta_k \rbrace$ are bounded and
    $\liminf_{k \to \infty} \inlinenorm{ \dot F(\theta_k)}_2 = 0$.
    \item If $\cup_{k\in\mathbb N} C_k$ is unbounded, $\{\delta_k : k +1 \in \mathbb N\}$ is bounded from below, and 
    \begin{equation}
      \liminf_{k\to\infty} \underline\alpha(C_k) \underline g(C_k) > 0,
    \end{equation}
    then $\liminf_{k \to \infty} \inlinenorm{ \dot F(\theta_k)}_2 = 0$.
    \item If $\cup_{k\in\mathbb N} C_k$ is unbounded, $\liminf_{k \to \infty} \delta_k = 0$,
    \begin{equation}
    \liminf_{u \to \infty} \frac{\underline g (C_u')^2 \underline \alpha (C_u')}{\overline g(C_u')^2 \overline \alpha (C_u')} > 0, \text{\quad where \quad} C_u' = \bigcup_{k=\ell_{g_{u-1}}}^{\ell_{g_u}-1} C_k,
    \end{equation}
    and, for some $w_1 \geq 0$ and $w_2 \in [0,2]$, $\exists c_0,c_1, c_2 \geq 0$ such that $\mathcal{L}(C_u') \leq c_0 + c_1 (F(\theta_{\ell_{g_{u-1}}}) - F_{l.b.})^{w_1} + c_2 \inlinenorm{ \dot F(\theta_{\ell_{g_u-1}})}_2^{w_2}$, then $\liminf_{k \to \infty} \inlinenorm{ \dot F(\theta_k)}_2 = 0$.
    \end{remunerate}
\end{remunerate} 
\end{theorem}
\begin{proof}
  By \cref{result-objective-function}, either the algorithm reaches a stationary point, or there are infinite iterates, and all elements of $\{o_s:s+1\in\mathbb{N}\}$ and $\{g_u:u+1\in\mathbb{N}\}$ are finite. We proceed by proving parts 2a, 2b, then 2c. Considering the first case, when $\cup_{k \in \mathbb N} C_k$ is bounded, then there exists a compact set $C$ such that $\cup_{k \in \mathbb{N}} C_k \subset C$. Since $\theta_k \in C_k$, $\limsup_k \inlinenorm{ \theta_k}_2 < \infty$, and thus $\lbrace \theta_k \rbrace$ are bounded. To show $\liminf_{k} \inlinenorm{\dot F(\theta_{k})}_2^2 = 0$, note that $\cup_{k \in \mathbb N} C_k$ being bounded implies $\liminf_{k \to \infty} \delta_k > 0$. Therefore, there exists a $\underline \delta > 0$ and $K \in \mathbb{N}$ such that $\delta_k > \underline \delta, ~\forall k \geq K$. Applying the proof of part 2 of \cref{result-zoutendjik}, and using that $0<\underline\alpha (C)\underline{g}(C) \leq \underline\alpha (C_{\ell_{g_u-1}}) \underline{g}(C_{\ell_{g_u-1}}), ~\forall u \in \mathbb N$, we conclude that
  \begin{equation}
    \underline\alpha (C)
    \underline{g}(C)
    \sum_{u:\ell_{g_u-1} \geq K}
    \inlinenorm{\dot F(\theta_{\ell_{g_u-1}})}_2^2 \leq
    \sum_{u:\ell_{g_u-1} \geq K}
    \underline\alpha (C_{\ell_{g_u-1}})
    \underline{g}(C_{\ell_{g_u-1}})
    \inlinenorm{\dot F(\theta_{\ell_{g_u-1}})}_2^2 
    < \infty.
  \end{equation}
  This implies $\liminf_{k} \inlinenorm{\dot F(\theta_{k})}_2^2 = 0$. Now consider the case when $\cup_{k \in \mathbb N} C_k$ is unbounded. For the first of these two cases (part 2b), we conclude directly by \cref{result-zoutendjik} that
  \begin{equation}
    \sum_{u=1}^\infty
    \underline\alpha (C_{\ell_{g_u-1}})
    \underline{g}(C_{\ell_{g_u-1}})
    \inlinenorm{\dot F(\theta_{\ell_{g_u-1}})}_2^2 
    < \infty.
  \end{equation}
  By the assumptions stated in the theorem, there exists a $\kappa > 0$ and $K \in \mathbb{N}$ such that 
  \begin{equation}
    \underline\alpha (C_{k})
    \underline{g}(C_{k}) \geq \kappa, ~\forall k \geq K.
  \end{equation}
  This in turn implies that $\sum_{u:\ell_{g_u -1} \geq K} \inlinenorm{\dot F(\theta_{\ell_{g_u-1}})}_2^2 < \infty$. Hence, $\liminf_{k} \inlinenorm{\dot F(\theta_{k})}_2^2 = 0$.
  
  Lastly, in the third case there exists a (different) $\kappa > 0$, and $U \in \mathbb{N}$, such that
  \begin{equation}
    \frac{\underline g (C_u')^2 \underline \alpha (C_u')}{\overline g(C_u')^2 \overline \alpha (C_u')} \geq \kappa, ~\forall u \geq U.
  \end{equation}
  Moreover, since $F(\theta_k) \leq F(\theta_0)$ by \cref{result-objective-function}, $\exists c_0' > 0$ such that $\mathcal{L}(C_u') \leq c_0' + c_2 \inlinenorm{ \dot F(\theta_{\ell_{g_u-1}})}_2^{w_2}$ for all $u \in \mathbb N$. Using these two facts and part 3 of \cref{result-zoutendjik}, there exists an index set $\mathcal U \subseteq \mathbb N$ such that 
  \begin{equation}
    \sum_{u \in \mathcal U, ~ u \geq U} \frac{\inlinenorm{ \dot F(\theta_{\ell_{g_u-1}})}_2^2}{c_0' + c_2 \inlinenorm{ \dot F(\theta_{\ell_{g_u-1}})}_2^{w_2} } < \infty.
  \end{equation}
  Now, if $c_2 = 0$ the result follows. Suppose $c_2 > 0$. We now verify that $\limsup_{u\in\mathcal U} \inlinenorm{ \dot F(\theta_{\ell_{g_u-1}})}_2 < \infty$ by contradiction, which will yield an upper bound for the denominator in the preceding inequality and, in turn, imply $\lim_{u\in\mathcal U} \inlinenorm{ \dot F(\theta_{\ell_{g_u-1}})}_2 = 0$. To this end, suppose $\limsup_{u\in\mathcal U} \inlinenorm{ \dot F(\theta_{\ell_{g_u-1}})}_2 = \infty$. Then, there exists a subsequence $\mathcal U' \subset \mathcal U$ such that $\forall u' \in \mathcal U' \Rightarrow 0$
  \begin{equation}
    0 < \frac{c_0'}{2 c_2} \leq \inlinenorm{ \dot F(\theta_{\ell_{g_{u'}-1}})}_2^2 - \frac{1}{2}\inlinenorm{ \dot F(\theta_{\ell_{g_{u'}-1}})}_2^{w_2} \Rightarrow 0 < \frac{1}{2 c_2} \leq \frac{\inlinenorm{ \dot F(\theta_{\ell_{g_{u'}-1}})}_2^2}{c_0' + c_2 \inlinenorm{ \dot F(\theta_{\ell_{g_{u'}-1}})}_2^{w_2} }.
  \end{equation} 
  Hence, we arrive at the contradiction, $\sum_{u \in \mathcal U', ~ u \geq U} (2c_2)^{-1} < \infty$. The result follows.
\end{proof}

\section{A Novel Step Size Procedure} \label{sec:StepSize}
Having introduced our general framework and discussed its theoretical properties in \cref{sec:Algo}, we present a well-performing instance (see \cref{sec:Results-cutest,sec:Results-gee}) of our framework that is equipped with a novel step size technique using negative gradient directions. The instance is specified by \cref{alg:novel-step-size,alg:update}, and there are two differences between \cref{alg:novel-step-size} and our general method. First, \cref{alg:novel-step-size} has specific parameters and subroutines. Second, \cref{alg:novel-step-size} includes an additional ``else-if'', which is superficial: it seperates the case when the gradient is above the upper threshold, and is exactly the same as accepting iterates normally, except for the addition of gradient threshold updating. We now describe our novel step size routine, and then provide some specific theory for this procedure.

\begin{algorithm}[!ht]
    \caption{Novel Step Size Method Applied within Our Framework}
    \label{alg:novel-step-size}
    \begin{algorithmic}[1]
    \Require $F, \dot F, \theta_{0}, \epsilon > 0$    
     
    \State $k \leftarrow 0$ \Comment{Outer loop counter}
   
    \State $\delta_k \leftarrow 1$ \Comment{Step size scaling}

    \State $w \leftarrow 10$ \Comment{Number of objective values used in nonmonotone search}

    \State $\tau_{\obj}^0 \leftarrow F(\theta_0)$ \Comment{Nonmontone search threshold}
    \State $\tau_{\gra,\low}^0, \tau_{\gra,\upp}^0 \leftarrow \inlinenorm{\dot{F}(\theta_0)}/\sqrt{2}, \sqrt{20}\tau_{\gra,\low}^0$ \Comment{Test thresholds on gradient}
    
    \While{$\inlinenorm{ \dot F(\theta_k) } > \epsilon$} \Comment{Outer loop}
    
    \State $j, \psi_0^k \leftarrow 0, \theta_k$ \Comment{Inner loop counter and initialization}
    
    \While{$\true$} \Comment{Inner loop}
            \State $\hat{L}_j^k \leftarrow \Call{Update}{j, k}$ \Comment{Helper function; Update local Lipschitz approximation}
    		\State $\alpha^k_{j} \leftarrow 
            \min\bigpar{ \frac{(\tau^{k}_{\gra,\low})^2}{\inlinenorm{\dot{F}(\psi_j^k)}_2^3 + .5\inlinenorm{\dot{F}(\psi_j^k)}_2^2 \hat{L}_j^k + 10^{-16 }}, \frac{1}{ \inlinenorm{\dot{F}(\psi_j^k)}_2 + .5 \hat{L}_j^k + 10^{-16} }} + 10^{-16}$ \Comment{Novel step size}
    		
    		\If{ $\inlinenorm{ \psi_j^k - \theta_k }_2 > 10$ 
    			\textbf{or} $\inlinenorm{ \dot F(\psi_j^k)}_2 \not\in (\tau_{\gra,\low}^k, \tau^k_{\gra,\upp})$ 
    			\textbf{or} $j == 100$} 
    			\If{ $F(\psi_j^k) \geq \tau_{\obj}^k - 10^{-4} \delta_k \alpha_0^k \inlinenorm{\dot{F}(\theta_k)}^2 $ }
    			\Comment{Fails nonmonotone Armijo condition}
    				\State $\theta_{k+1}, \delta_{k+1} \leftarrow \theta_k, .5 \delta_k$
    				\Comment{Reset iterate with reduced step size scaling}
    			\ElsIf{ $\inlinenorm{ \dot F(\psi^k_j)}_2 \leq \tau_{\gra,\low}^k$ }
    				\State $\theta_{k+1}, \delta_{k+1} \leftarrow \psi_j^k, \delta_k$
				\Comment{Accept iterate and leave step size unchanged}    			
    				\State Set $\tau_{\obj}^{k+1} \leftarrow$ by \cref{eqn-nonmonotone-threshold}
                
				    \State $\tau_{\gra,\low}^{k+1}, \tau_{\gra,\upp}^{k+1} \leftarrow \inlinenorm{\dot{F}(\theta_{k+1})}/\sqrt{2}, \sqrt{20}\tau_{\gra,\low}^{k+1}$
					\Comment{Reset gradient thresholds}
				\ElsIf{$\inlinenorm{ \dot F(\psi^k_j)}_2 \geq \tau_{\gra,\upp}^k$}
					\State $\theta_{k+1}, \delta_{k+1}, \leftarrow \psi_j^k, \min\{1.5\delta_k, 1\}$
					\Comment{Accept iterate and increase step size}
    				\State Set $\tau_{\obj}^{k+1} \leftarrow$ by \cref{eqn-nonmonotone-threshold}
					\State $\tau_{\gra,\low}^{k+1}, \tau_{\gra,\upp}^{k+1} \leftarrow\inlinenorm{\dot{F}(\theta_{k+1})}/\sqrt{2}, \sqrt{20}\tau_{\gra,\low}^{k+1}$
					\Comment{Reset gradient thresholds}
    			\Else
    				\State $\theta_{k+1}, \delta_{k+1}, \leftarrow \psi_j^k, \min\{1.5\delta_k, 1\}$
    				\Comment{Accept iterate and increase step size}
    				\State Set $\tau_{\obj}^{k+1} \leftarrow$ by \cref{eqn-nonmonotone-threshold}
					\State $\tau_{\gra,\low}^{k+1}, \tau_{\gra,\upp}^{k+1} \leftarrow \tau_{\gra,\low}^{k} \tau_{\gra,\upp}^{k} $
					\Comment{No gradient violation so use past interval}
    			\EndIf
			\State $k \leftarrow k+1$
    			\State Exit Inner Loop
    		\EndIf
    		
    		\State $\psi_{j+1}^k, j \leftarrow \psi_j^k - \delta_k \alpha_j^k \dot{F}(\psi^k_j), j+1$
    		\Comment{Standard gradient descent}
    
    \EndWhile 

    \EndWhile 
	\State \Return $\theta_k$    
    \end{algorithmic}
    
\end{algorithm} 

\begin{algorithm}[!ht]
    \caption{Update Scheme for Local Lipschitz Approximation}
    \label{alg:update}
    \begin{algorithmic}[1]
        \Require $\dot{F}, \psi_{j}^k, \psi_{j-1}^k, L(k)$
        \Function{Update}{$j,k$} \Comment{Method for Local Lipschitz Approximation}
            \If{$j == 0$ \textbf{and} $k == 0$}
                \State $\hat{L}_0^0 \leftarrow 1$
                \Comment{Initial value}
            \ElsIf{$j == 0$ \textbf{and} $k > 0$}
                \State $\hat{L}_0^k \leftarrow \hat{L}_{j_{k-1}}^{k-1}$
                \Comment{Use previous estimate in subsequent initial inner loop iterations}
            \ElsIf{$j > 0$ \textbf{and} $k \ge 0$ \textbf{and} $L(k)=k$}
                \State $\hat{L}_{j}^k \leftarrow \frac{\inlinenorm{\dot{F}(\psi_j^k) - \dot{F}(\psi_{j-1}^k)}_2}{\inlinenorm{\psi_j^k-\psi_{j-1}^k}_2} $
                \Comment{Aggressive estimate when $\psi^{k-1}_{j_{k-1}}$ was accepted}
            \ElsIf{$j > 0$ \textbf{and} $k \geq 0$ \textbf{and} $L(k)\not=k$}
                \State $\hat{L}_{j}^k \leftarrow \max\bigpar{ \frac{\inlinenorm{\dot{F}(\psi_j^k) - \dot{F}(\psi_{j-1}^k)}_2}{\inlinenorm{\psi_j^k-\psi_{j-1}^k}_2}, \hat{L}_{j-1}^k }$
                \Comment{Conservative estimate when $\psi^{k-1}_{j_{k-1}}$ was rejected}
            \EndIf
            \State \Return $\hat{L}_j^k$
        \EndFunction
    \end{algorithmic}
\end{algorithm}

\subsection{Novel Step Size Routine Description} Our step size method uses a local Lipschitz approximation to perform the update. We first contextualize our local Lipschitz approximation strategy (Line 9 of \cref{alg:novel-step-size}) and then discuss how it is used to compute the step size (Line 10 of \cref{alg:novel-step-size}).

The local Lipschitz approximation is calculated in \cref{alg:update}, and has four cases: an initialization phase, re-initialization phase for subsequent initial inner loop iterations, an aggressive phase if the current outer-loop iterate, $\theta_k$, was accepted, and a conservative phase if the previous outer-loop iterate was rejected. While similar local Lipschitz approximation ideas exist \cite{nesterov2012GradientMF,curtis2018exploiting,malitsky2020adaptive,zhang2020firstorder,malitsky2023adaptive}, our method differs in that the approximation is used irrespective of any local model in the region, eliminating the need for objective evaluations to verify the accuracy of such a model. \textit{We emphasize that we do not explicitly assume any type of global behavior, nor knowledge of the local Lipschitz constant in the algorithm}.

Using the local Lipschitz approximation, the step size calculation occurs (see Line 10 of \cref{alg:novel-step-size}). The step size is selected as the minimum of two quantities: the first quantity is inspired by leveraging our first triggering event and Zoutendjik's analysis method to ensure descent, while the second ensures nice theoretical properties (see \cref{lemma:novel-step-size-satisfies-properties}). The constants $10^{-16}$ appear for numerical stability.

\paragraph{Theory} We specialize the theory from \cref{sec:Algo} to \cref{alg:novel-step-size}. We verify that \cref{alg:novel-step-size} satisfies the subroutine properties in \cref{lemma:novel-step-size-satisfies-properties}, which will allow us to apply our previous results. The convergence analysis then follows by utilizing \Cref{result-gradient}, which just requires proving a property of our novel step size, and verifying that there exists compact sets that bound the inner loop iterates. What follows is an example of how to use \Cref{result-gradient}, and provides a blueprint for analyzing other instances of our framework.

We now prove that the subroutines and parameters in \cref{alg:novel-step-size} satisfy \cref{property-stepdirection-negative,property-stepdirection-propgrad,property-stepsize-upper,property-stepsize-lower,property-radius-nonnegative,property-radius-bounded,property-grad-lower,property-grad-upper,property-iter-max}.

\begin{lemma} \label{lemma:novel-step-size-satisfies-properties}
    \Cref{alg:novel-step-size} has subroutines $\mathrm{StepDirection()}$ and $\mathrm{StepSize()}{}$ that satisfy \cref{property-stepdirection-negative,property-stepdirection-propgrad,property-stepsize-upper,property-stepsize-lower} and parameters $\tau_{\iter, \exit}^k, \tau^k_{\gra,\low}, \tau^k_{\gra, \upp}, \tau_{\iter, \max}^k$ that satisfy \cref{property-radius-nonnegative,property-radius-bounded,property-grad-lower,property-grad-upper,property-iter-max}.
\end{lemma}
\begin{proof}
    Take any compact set $C \subset \mathbb{R}^n$. In \cref{alg:novel-step-size}, $\mathrm{StepDirection(\psi)}$ returns the negative gradient direction at $\psi$, therefore satisfies \cref{property-stepdirection-negative,property-stepdirection-propgrad} with $\underbar{$g$}(C) = \overline{g}(C) = 1$. Next to prove \cref{property-stepsize-upper,property-stepsize-lower}, recall we must show that $\forall \psi \in C$, any $\alpha$ computed by $\mathrm{StepSize()}{}$ at $\psi$ must satisfy $\underbar{$\alpha$}(C)  < \alpha < \overline{\alpha}(C)$ for some constants $\overline{\alpha}(C), \underbar{$\alpha$}(C) > 0$. From \cref{alg:novel-step-size}, $\mathrm{StepSize()}{}$ returns for any $\psi$
    \begin{equation}
        \alpha = \min\bigpar{ \frac{(\tau^k_{\gra,\low})^2}{\inlinenorm{\dot{F}(\psi)}_2^3 + .5\inlinenorm{\dot{F}(\psi)}_2^2 \hat{L}_j^k + 10^{-16 }}, \frac{1}{ \inlinenorm{\dot{F}(\psi)}_2 + .5 \hat{L}_j^k + 10^{-16} }} + 10^{-16},
    \end{equation}
    where $\hat{L}_j^k \ge 0$ and $(\tau^k_{\gra,\low})^2 \ge 0$. Since everything inside the minimum is non-negative, $\underbar{$\alpha$}(C) \geq 10^{-16}$. To upper bound $\alpha$,
    \begin{equation}
        \alpha \leq \frac{1}{ \inlinenorm{\dot{F}(\psi)}_2 + .5 \hat{L}_j^k + 10^{-16} } + 10^{-16} \le \frac{1}{10^{-16}} + 10^{-16} = \overline{\alpha}(C).
    \end{equation}
    To prove our other parameters satisfy the required properties, note that $\tau_{\iter, \exit}^k = 10$, which is non-negative and bounded above (\cref{property-radius-nonnegative,property-radius-bounded}); and $\tau_{\iter, \max}^k = 100$, which is greater than 1 and bounded above (\cref{property-iter-max}). Lastly, in \cref{alg:novel-step-size}, we iteratively define the gradient upper and lower bound whenever the gradient thresholds are violated as
    \begin{equation}
        \begin{aligned}
            0 < \tau_{\gra,\low}^k = \inlinenorm{\dot{F}(\theta_k)}_2/\sqrt{2} < \inlinenorm{\dot{F}(\theta_k)}_2, \text{\quad}
            \inlinenorm{\dot{F}(\theta_k)}_2 < \tau_{\gra,\upp}^k = \sqrt{10}\inlinenorm{\dot{F}(\theta_k)}_2.
        \end{aligned}
    \end{equation}
    Whenever these parameters are not updated, the gradient is strictly within the previous interval, which satisfy \cref{property-grad-lower,property-grad-upper}.
\end{proof}

\begin{theorem} \label{thm:convergence-theorem-novel-step}
    Suppose \cref{problem} satisfies \cref{as-bounded below,as-loc-lip-cont}, and is solved using \cref{alg:novel-step-size}. Then, either there exists a $k+1 \in \mathbb{N}$ such that $F(\theta_k) \le F(\theta_0)$ and $\dot{F}(\theta_k)=0$, or $\{F(\theta_k) : k+1 \in \mathbb{N}\}$ is bounded and has a strictly decreasing subsequence. In the latter case, we have the following cases.
    \begin{remunerate}
        \item If $\{\theta_k\}$ are bounded, then $\liminf_{k\to\infty}\inlinenorm{ \dot F(\theta_k) }_2 = 0$
        \item If $\{\theta_k\}$ are unbounded and for some $w_1 \geq 0$ and $w_2 \in [0,2]$, $\exists c_0,c_1, c_2 \geq 0$ such that $\mathcal{L}(C_u') \leq c_0 + c_1 (F(\theta_{\ell_{g_{u-1}}}) - F_{l.b.})^{w_1} + c_2 \inlinenorm{ \dot F(\theta_{\ell_{g_u-1}})}_2^{w_2}$, where $C_u' = \cup_{k = \ell_{g_{u-1}}}^{\ell_{g_u}-1} C_k$, then $\liminf_{k \to \infty} \inlinenorm{ \dot F(\theta_k)}_2 = 0$.
    \end{remunerate}
\end{theorem}
\begin{proof}
    By \cref{lemma:novel-step-size-satisfies-properties} and \cref{result-objective-function}, either the procedure terminates in finite time or produces an infinite number of iterates with bounded objective function values (containing a strictly decreasing subsequence). We study the sub-cases of this latter case now. When $\{\theta_k\}$ is bounded, we must show that there exists a compact $C$ that contains $\{\theta_k\}$ and all iterates $\{\psi_1^k,...,\psi_{j_k}^k\}$. To this end, by our parameter initializations and construction of the triggering events, for any $k+1 \in \mathbb{N}$, define $\mathcal{B}(\theta_k, 10) = \{\psi : \inlinenorm{\psi-\theta_k}_2 \le 10\}$, then $\{\psi_1^k,...,\psi_{j_k-1}^k\} \subset \mathcal{B}(\theta_k, 10)$. To bound the distance between $\psi_{j_k}^k$ and $\psi_{j_k-1}^k$, define $\mathcal{G}(\theta, R) = \sup_{\psi \in \mathcal{B}(\theta, R)} \inlinenorm{\dot{F}(\psi)}_2$, then using that $\delta_k \le 1$ and $\alpha_j^k \le 1/(\inlinenorm{\dot{F}(\psi_{j_k-1}^k)}_2 + .5 \hat{L}_{j_k-1}^k + 10^{-16}) + 10^{-16}$, we obtain
    \begin{equation}
        \norm{\psi_{j_k}^k - \psi_{j_k-1}^k}_2 = \norm{\delta_k \alpha_{j}^k \dot{F}(\psi_{j_k-1}^k)}_2 \le 1 + 10^{-16} \norm{\dot{F}(\psi_{j_k-1}^k)}_2 \le 1 + 10^{-16}\mathcal{G}(\theta_k, 10).
    \end{equation}
    Therefore, $\norm{\theta_k - \psi_{j_k}^k}_2 \le 11 + 10^{-16}\mathcal{G}(\theta_k, 10)$.
    Define $C_k = \{\psi : \inlinenorm{\psi-\theta_k}_2 \le 11 + 10^{-16}\mathcal{G}(\theta_k, 10)\}$, then for all $k+1 \in \mathbb{N}$, $\{\psi_1^k,...,\psi_{j_k}^k\} \subset C_k$. By our assumptions, $\exists R > 0$ such that for all $k+1 \in\mathbb{N}, ~\inlinenorm{\theta_k}_2 \leq R$. Therefore, for all $k+1 \in\mathbb{N}$ it must be the case that $\mathcal{G}(\theta_k, 10) \le \mathcal{G}(0, 10+R)$. Using this fact, we conclude
    \begin{equation}
        \bigcup_{k+1 \in \mathbb{N}} C_k \subseteq \{\psi:\inlinenorm{\psi}_2 \le 11 + 10^{-16}\mathcal{G}(0, 10+R) \} \eqdef C.
    \end{equation}
    Therefore, by part 2a of \cref{result-gradient}, $\liminf_{k\to\infty}\inlinenorm{ \dot F(\theta_k) }_2 = 0$. 
    
    Now consider when $\{\theta_k\}$ is unbounded, then $\cup_{k\in\mathbb N} C_k$ is unbounded. Note, either $\lbrace \delta_k : k+1 \in \mathbb{N} \rbrace$ is bounded below, or there exists a subsequence converging to zero (i.e. $\liminf_{k\to\infty} \delta_k = 0$). Therefore, to use part either 2b or 2c of \cref{result-gradient}, we verify 
    \begin{equation} \label{eq-inf-for-novel-step}
         \liminf_{k\to\infty} \underline g (C_k) \underline \alpha (C_k) > 0 \text{\quad and \quad}
         \liminf_{u\to\infty} \frac{\underline g (C_u')^2 \underline \alpha (C_u')}{\overline g(C_u')^2 \overline \alpha (C_u')} > 0.
    \end{equation}
    This follows since for any compact set $C$, the negative gradient directions satisfy \cref{property-stepdirection-negative,property-stepdirection-propgrad} with $\underline g (C) = \overline g(C) = 1$, and by \cref{lemma:novel-step-size-satisfies-properties}, $\underline\alpha(C_k)$ and $\overline \alpha(C_k)$ are bounded by non-zero constants.
\end{proof}

In summary, we have just presented a novel step size procedure utilizing negative gradient directions within our framework in \cref{alg:novel-step-size}, and shown how to apply the theory developed for our general method to its specific subroutines and parameters. We now numerically show that this procedure is very competitive against a range of first and second order methods on CUTEst problems (\cref{sec:Results-cutest}), and has superior performance when comparing on several data science problems (\cref{sec:Results-gee}).

\section{Numerical Experiments on CUTEst Problems}  \label{sec:Results-cutest}
We now present a numerical experiment using our methodology (see \cref{alg:novel-step-size}) on a set of unconstrained optimization problems from the CUTEst library \cite{gould2015}. The details of the experiment are in \cref{table:cutest-exper}.

\begingroup
    \setlength{\tabcolsep}{10pt} 
    \renewcommand{\arraystretch}{1.5}
    \begin{table}[!ht]
        \centering
        \footnotesize
        \caption{CUTEst numerical experiment overview.}
        \label{table:cutest-exper}
        \begin{tabular}{p{4cm} p{9cm}} \toprule
            \textbf{Problems} & All\cref{footnote-alg-exception} unconstrained CUTEst problems with objective, gradient, and Hessian information on the smallest dimension setting (117 problems). \\
            \textbf{Algorithms} & Gradient Descent with Armijo and Wolfe line search \cite[Chapter 3]{nocedal2006numerical}, Cubic Regularization Newton's (CRN, \cite{nesterov2006cubic}), Adaptive Cubic Regularization (ACR, \cite{cartis2011cubic1}), Dynamic Method (DMC, \cite{curtis2018exploiting}), and Our method (\cref{alg:novel-step-size}).  \\
            \textbf{Termination} & $20{,}000$ iterations, or gradient tolerance of $10^{-5}$. \\
            \textbf{Data Recorded} & Number of objective and gradient evaluations, and CPU Time. \\\bottomrule
        \end{tabular}
    \end{table}
\endgroup


To summarize the information in \cref{table:cutest-exper}, we run three first order methods and three second order methods for $20{,}000$ iterations, or until a gradient tolerance of $10^{-5}$ is reached, on all unconstrained problems that have objective, gradient, and Hessian information.\footnote{\label{footnote-alg-exception}Except for SCURLY10, SCULRY20, SCULRY30, and TESTQUAD as the first three gave initialization errors, and the last due to computational time constraints.}
For algorithms that have an optimization sub-problem to compute the next iterate, we utilize a Krylov-based trust region method accessed through SciPy \cite[see][]{virtanen2020scipy,gould1999trustregion}, with a limit of $500$ iterations, or until a specific gradient tolerance is reached \cite[see][\S 7 for such tolerances]{cartis2011cubic1}.\footnote{Necessary for CRN, ACR}
All methods that have some type of inner loop are limited to $100$ inner loop iterations.\footnote{Necessary for Line search, CRN, DMC, Ours} 
Finally, to compare these methods we record the number of objective, gradient, and Hessian evaluations, as well as CPU time; in our analysis we concentrate on comparing these quantities only on problems where \emph{all} algorithms reached an iterate with gradient $10^{-5}$ (i.e., a successful termination) leaving a total of 76 problems in the analysis.

Before presenting the results, we remark that the list of algorithms in \cref{table:cutest-exper} (as well as in \cref{sec:Results-gee}) is missing a class of methods that have recently gained attention called Objective Function Free Optimization algorithms \cite{gratton2022firstorderOFFO,gratton2023OFFORegularized,gratton2023OFFOSecondOrderOpt,gratton2022ParametricComplexityAnalysis,gratton2023ComplexityOfFirstOrderOFFO}. We leave these algorithms out of our comparison, because many assumptions required to guarantee convergence are not satisfied by problems in data science; specifically, these methods require global Lipschitz continuity of the gradient, and that $\forall \theta \in \mathbb{R}^n, ~\inlinenorm{F(\theta)}_{\infty} < M$ for some $M \in \mathbb{R}$ \cite[see][AS.2 and AS.3]{gratton2022firstorderOFFO}, which are both unrealistic \cite[see][\S 5]{gratton2022firstorderOFFO}.

Moving to the results, we consider the relative change of objective and gradient evaluations between our method and the comparing algorithms \cite[as was done in][]{curtis2018exploiting}, then compare the CPU Time to examine any computational overhead besides explicit oracle evaluations.


\paragraph{Comparison with First Order Methods} The relative change of objective and gradient evaluations between our method and the first order methods are presented in \cref{fig:first-order-rel-change-total}. As illustrated in \cref{fig:first-order-rel-change-total}, our method is extremely competitive in these terms, as it uses fewer combined objective and gradient evaluations on all 76 problems except for 5. When comparing just objective evaluations in \cref{fig:first-order-rel-change-func-grad}, our method uses evaluations economically owing to the triggering events for the inner loop, whereas line search methods rely heavily on repeated objective evaluations. On the other hand, our method requires more gradient evaluations on many problems, as some inner loop iterations are needed to calibrate the scaling factor.

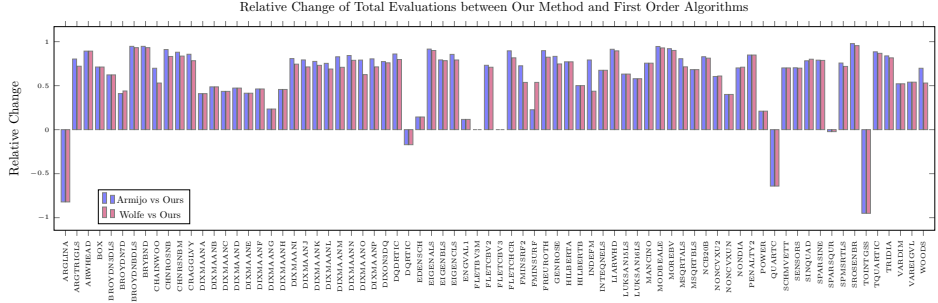
\begin{figure}
    \centering
    \begin{tikzpicture}[scale = .5]

    \pgfplotstableread[col sep=comma]{data/cutest/all_rel_change_total_calls.csv}\reldata
    \pgfplotsset{every tick label/.append style={font=\tiny}}

    \begin{axis}[ybar = .5pt, 
        xtick = data,
        xticklabels from table={\reldata}{problem name},
        xticklabel style={rotate=90},
        enlarge x limits={abs=1},
        bar width = 3pt,
        width=25cm,height=7cm,
        ylabel = Relative Change,
        title = Relative Change of Total Evaluations between Our Method and First Order Algorithms,
        legend style={at={(0.05,0.2)},anchor=north west,nodes={scale=0.75, transform shape}}]

        \addplot[draw=black, fill=blue, semitransparent] table [x expr = \coordindex, y = armijo_rel_change] {\reldata};
        \addplot[draw=black, fill=purple, semitransparent] table[x expr = \coordindex, y = wolfe_rel_change] {\reldata};
        \legend{Armijo vs Ours, Wolfe vs Ours}
    \end{axis}
\end{tikzpicture} 
    \caption{Relative change of total (objective plus gradient) evaluations between Our Method and Gradient Descent with Armijo and Wolfe line search. Negative values indicate ours did worse, positive values indicate ours did better. \label{fig:first-order-rel-change-total}}
\end{figure}

\begin{figure}
    \centering
    \begin{tikzpicture}[scale = .5]

    \pgfplotstableread[col sep=comma]{data/cutest/all_rel_change_gradient_calls.csv}\reldata
    \pgfplotsset{every tick label/.append style={font=\tiny}}

    \begin{axis}[ybar = .5pt, 
        xtick = data,
        xticklabels from table={\reldata}{problem name},
        xticklabel style={rotate=90},
        enlarge x limits={abs=1},
        bar width = 3pt,
        width=25cm,height=7cm,
        ylabel = Relative Change,
        title = Relative Change of Gradient Evaluations between Our Method and First Order Algorithms,
        legend style={at={(0.05,0.2)},anchor=north west,nodes={scale=0.75, transform shape}}]

        \addplot[draw=black, fill=blue, semitransparent] table [x expr = \coordindex, y = armijo_rel_change] {\reldata};
        \addplot[draw=black, fill=purple, semitransparent] table[x expr = \coordindex, y = wolfe_rel_change] {\reldata};
        \legend{Armijo vs Ours, Wolfe vs Ours}
    \end{axis}
\end{tikzpicture}
    \begin{tikzpicture}[scale = .5]

    \pgfplotstableread[col sep=comma]{data/cutest/all_rel_change_function_calls.csv}\reldata
    \pgfplotsset{every tick label/.append style={font=\tiny}}

    \begin{axis}[ybar = .5pt, 
        xtick = data,
        xticklabels from table={\reldata}{problem name},
        xticklabel style={rotate=90},
        enlarge x limits={abs=1},
        bar width = 3pt,
        width=25cm,height=7cm,
        ylabel = Relative Change,
        title = Relative Change of Objective Evaluations between Our Method and First Order Algorithms,
        legend style={at={(0.05,0.2)},anchor=north west,nodes={scale=0.75, transform shape}}]

        \addplot[draw=black, fill=blue, semitransparent] table [x expr = \coordindex, y = armijo_rel_change] {\reldata};
        \addplot[draw=black, fill=purple, semitransparent] table[x expr = \coordindex, y = wolfe_rel_change] {\reldata};
        \legend{Armijo vs Ours, Wolfe vs Ours}
    \end{axis}
\end{tikzpicture}
    \caption{Relative change of objective or gradient between Our Method and Gradient Descent with Armijo and Wolfe line search. Negative values indicate ours did worse, positive values indicate ours did better. \label{fig:first-order-rel-change-func-grad}}
\end{figure}
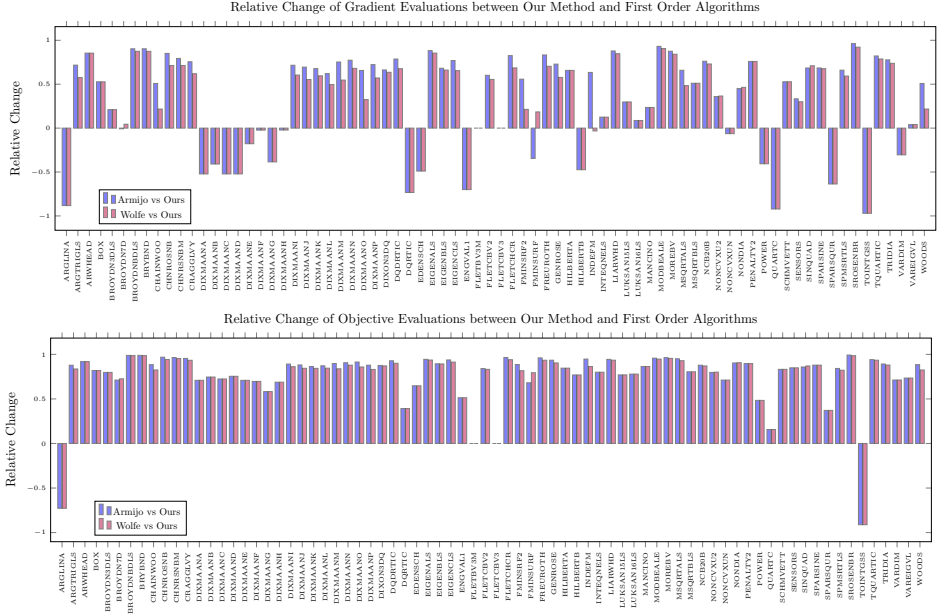

\begin{figure}
    \centering
    \begin{subfigure}{0.4\textwidth}
        \begin{tikzpicture}[x=1pt,y=1pt,scale=.6]
	\begin{axis}[xmode=log,ymode=log,
		xlabel = Our CPU Time,
		ylabel = Competitors CPU Time,
		title = CPU Time Comparison with First Order Methods,
		xmin=.00001,xmax=1,ymin=.00001,
		ymax=1,width=10cm,height=9cm,
		legend style={at={(0.02,0.98)},anchor=north west,nodes={scale=0.75, transform shape}}]
		
		\addplot [color=black, line width=1.0pt, forget plot]
		  table[row sep=crcr]{%
		  .00001  .00001\\
		1 1\\
		};

		\addplot[only marks, color = blue, mark options={scale=.75}] table [x=GD Triggers, y = armijo, col sep=comma] {data/cutest/all_cpu_lost.csv};
		\addlegendentry[]{Ours did better vs Armijo}
		\addplot[only marks, color = red, mark options={scale=.75}] table [x=GD Triggers, y = armijo, col sep=comma] {data/cutest/all_cpu_won.csv};
		\addlegendentry[]{Armijo did better vs Ours}
	
		\addplot[only marks, color = blue, mark = square,mark options={scale=.75}] table [x=GD Triggers, y = wolfe, col sep=comma] {data/cutest/all_cpu_lost.csv};
		\addlegendentry[]{Ours did better vs Wolfe}
		\addplot[only marks, color = red, mark = square,mark options={scale=.75}] table [x=GD Triggers, y = wolfe, col sep=comma] {data/cutest/all_cpu_won.csv};
		\addlegendentry[]{Wolfe did better vs Ours}
	\end{axis}
\end{tikzpicture}
        \caption{First order methods}
        \label{fig:first-order-cpu-time-comparison}
    \end{subfigure}
    \begin{subfigure}{0.4\textwidth}
\begin{tikzpicture}[x=1pt,y=1pt,scale=.6]
    \begin{axis}[xmode=log,ymode=log,
		xlabel = Our CPU Time,
		ylabel = Competitors CPU Time,
		title = CPU Time Comparison with Second Order Methods,
		xmin=.00001,xmax=1,ymin=.00001,
		ymax=1000,width=10cm,height=9cm,
        legend style={at={(0.02,0.98)},anchor=north west,nodes={scale=0.75, transform shape}}]
        
        \addplot [color=black, line width=1.0pt, forget plot]
		  table[row sep=crcr]{%
		  .00001  .00001\\
		1 1\\
		};
        
        \addplot[only marks, color = blue, mark options={scale=.75}] table [x=GD Triggers, y = Adaptive Cubic Regularization, col sep=comma] {data/cutest/all_cpu_lost.csv};
		\addlegendentry[]{Ours did better vs ACR}
		\addplot[only marks, color = red, mark options={scale=.75}] table [x=GD Triggers, y = Adaptive Cubic Regularization, col sep=comma] {data/cutest/all_cpu_won.csv};
		\addlegendentry[]{ACR did better vs Ours}

        \addplot[only marks, color = blue, mark = square,mark options={scale=.75}] table [x=GD Triggers, y = Cubic Regularized Newton, col sep=comma] {data/cutest/all_cpu_lost.csv};
		\addlegendentry[]{Ours did better vs CRN}
		\addplot[only marks, color = red, mark = square,mark options={scale=.75}] table [x=GD Triggers, y = Cubic Regularized Newton, col sep=comma] {data/cutest/all_cpu_won.csv};
		\addlegendentry[]{CRN did better vs Ours}

        \addplot[only marks, color = blue, mark = star,mark options={scale=.75}] table [x=GD Triggers, y = Dynamic Method: Curtis and Robinson, col sep=comma] {data/cutest/all_cpu_lost.csv};
        \addlegendentry[]{Ours did better vs DM}
        \addplot[only marks, color = red, mark = star, mark options={scale=.75}] table [x=GD Triggers, y = Dynamic Method: Curtis and Robinson, col sep=comma] {data/cutest/all_cpu_won.csv};
        \addlegendentry[]{DM did better vs Ours}
    \end{axis}
\end{tikzpicture}
        \caption{Second order methods}
        \label{fig:second-order-cpu-time-comparison}
    \end{subfigure}
    \caption{Comparison of CPU Time between our method and other algorithms on CUTEst problems split between first order and second order methods.}
\end{figure}
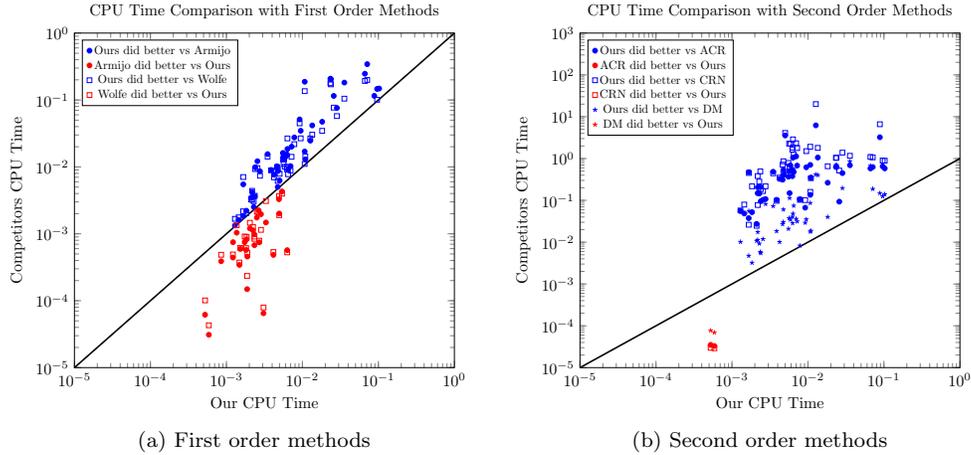

When comparing CPU time between our method and other first order methods, it can be seen from \cref{fig:first-order-cpu-time-comparison} that our method compares favorably against line search algorithms on only slightly larger than half the number of problems. While \cref{fig:first-order-rel-change-total,fig:first-order-rel-change-func-grad} might lead one believe that our method would be outright faster on all problems, interestingly there does seem to be some non-trivial computational aspects when checking triggering events in the inner loop or the extra gradient evaluations. 

\paragraph{Comparison with Second Order Methods} We present the same relative change graphics in \cref{fig:second-order-rel-change-total} and \cref{fig:second-order-rel-change-grad-func} to compare against second order methods. As can be seen from \cref{fig:second-order-rel-change-total}, cubic regularized Newton's method and the adaptive cubic regularization method require substantially fewer objective and gradient evaluations compared to our first-order algorithm. The same observation holds when comparing relative change of just the number of objective or the gradient evaluations separately against these two (see \cref{fig:second-order-rel-change-grad-func}). Such results are not surprising as second-order methods make use of Hessian information in order to find a minimizer \cite[see][]{nesterov2006cubic}, which is not used in our example procedure. When comparing CPU times, we see the cost of using Hessian information and computing solutions to sub-problems (see \cref{fig:second-order-cpu-time-comparison}), as our method does better on almost all problems. 

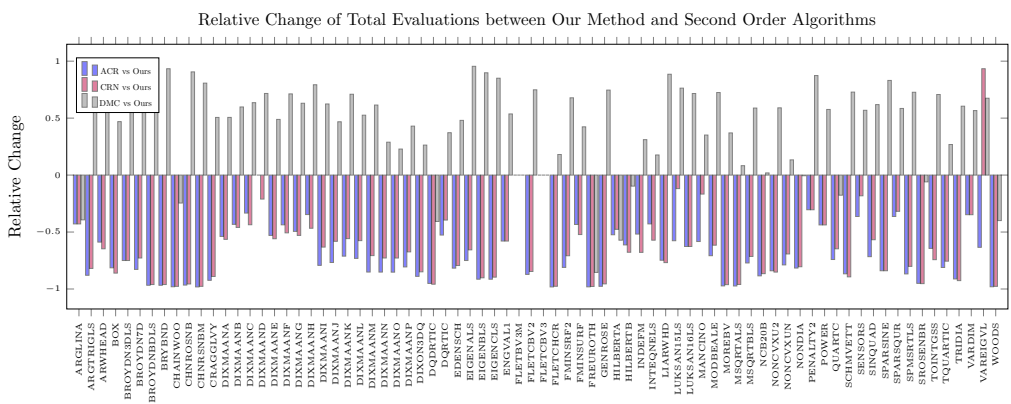
\begin{figure}
    \centering
    \begin{tikzpicture}[scale = .65]

    \pgfplotstableread[col sep=comma]{data/cutest/all_rel_change_total_calls.csv}\reldata
    \pgfplotsset{every tick label/.append style={font=\tiny}}


    \begin{axis}[ybar = .25pt, 
        x = .25cm,
        xtick = data,
        xticklabels from table={\reldata}{problem name},
        xticklabel style={rotate=90},
        enlarge x limits={abs=1},
        bar width = 2pt,
        width=25cm,height=7cm,
        ylabel = Relative Change,
        title = Relative Change of Total Evaluations between Our Method and Second Order Algorithms,
        legend style={at={(0.01,0.95)},anchor=north west,nodes={scale=0.5, transform shape}}]

        \addplot[draw=black, fill=blue, semitransparent] table [x expr = \coordindex, y = acr_rel_change] {\reldata};
        \addplot[draw=black, fill=purple, semitransparent] table[x expr = \coordindex, y = crn_rel_change] {\reldata};
        \addplot[draw=black, fill=gray, semitransparent] table[x expr = \coordindex, y = dmc_rel_change] {\reldata};
        \legend{ACR vs Ours, CRN vs Ours, DMC vs Ours}
    \end{axis}
\end{tikzpicture}
    \caption{Relative change of total (objective plus gradient) evaluations between our method and second order algorithms. Negative values indicate ours did worse, positive values indicate ours did better.\label{fig:second-order-rel-change-total}}
\end{figure}

\begin{figure}
    \centering
    \begin{tikzpicture}[scale = .65]

    \pgfplotstableread[col sep=comma]{data/cutest/all_rel_change_gradient_calls.csv}\reldata
    \pgfplotsset{every tick label/.append style={font=\tiny}}


    \begin{axis}[ybar = .25pt, 
        x = .25cm,
        xtick = data,
        xticklabels from table={\reldata}{problem name},
        xticklabel style={rotate=90},
        enlarge x limits={abs=1},
        bar width = 2pt,
        width=25cm,height=7cm,
        ylabel = Relative Change,
        title = Relative Change of Gradient Evaluations between Our Method and Second Order Algorithms,
        legend style={at={(0.01,0.95)},anchor=north west,nodes={scale=0.5, transform shape}}]

        \addplot[draw=black, fill=blue, semitransparent] table [x expr = \coordindex, y = acr_rel_change] {\reldata};
        \addplot[draw=black, fill=purple, semitransparent] table[x expr = \coordindex, y = crn_rel_change] {\reldata};
        \addplot[draw=black, fill=gray, semitransparent] table[x expr = \coordindex, y = dmc_rel_change] {\reldata};
        \legend{ACR vs Ours, CRN vs Ours, DMC vs Ours}
    \end{axis}
\end{tikzpicture}
    \begin{tikzpicture}[scale = .65]

    \pgfplotstableread[col sep=comma]{data/cutest/all_rel_change_function_calls.csv}\reldata
    \pgfplotsset{every tick label/.append style={font=\tiny}}


    \begin{axis}[ybar = .25pt, 
        x = .25cm,
        xtick = data,
        xticklabels from table={\reldata}{problem name},
        xticklabel style={rotate=90},
        enlarge x limits={abs=1},
        bar width = 2pt,
        width=25cm,height=7cm,
        ylabel = Relative Change,
        title = Relative Change of Objective Evaluations between Our Method and Second Order Algorithms,
        legend style={at={(0.01,0.95)},anchor=north west,nodes={scale=0.5, transform shape}}]

        \addplot[draw=black, fill=blue, semitransparent] table [x expr = \coordindex, y = acr_rel_change] {\reldata};
        \addplot[draw=black, fill=purple, semitransparent] table[x expr = \coordindex, y = crn_rel_change] {\reldata};
        \addplot[draw=black, fill=gray, semitransparent] table[x expr = \coordindex, y = dmc_rel_change] {\reldata};
        \legend{ACR vs Ours, CRN vs Ours, DMC vs Ours}
    \end{axis}
\end{tikzpicture}
    \caption{Relative change between our method and second order algorithms of gradient and objective evaluations, respectively. Negative values indicate ours did worse, positive values indicate ours did better.\label{fig:second-order-rel-change-grad-func}}
\end{figure}
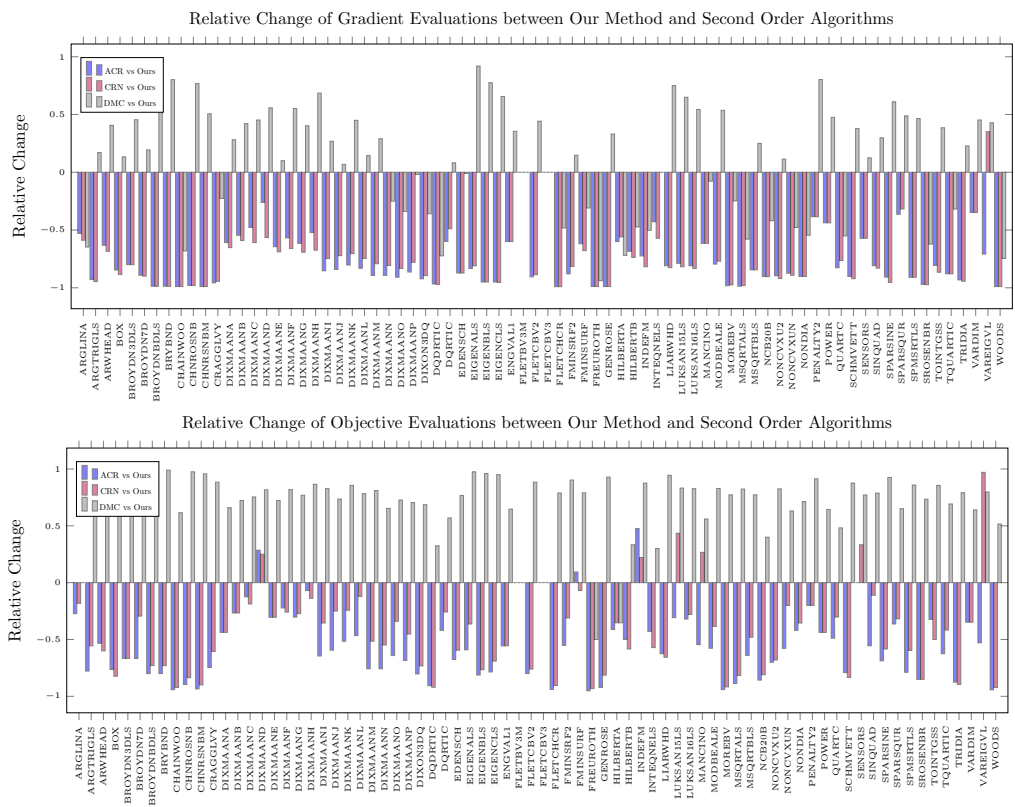

\paragraph{In Summary} When comparing our method against this set of first and second order methods on a set of CUTEst problems, we see that our method is extremely competitive. In comparison to first order methods, our novel procedure economically uses objective evaluations at the price of a small increase in gradient evaluations; while, in comparison to second order methods, our algorithm has superior CPU time performance even when taking more objective and gradient evaluations. This makes our method a competitive and practical alternative to traditional line search techniques and second order methods in addressing general unconstrained optimization problems, and with the possible exception of Armijo line search, is the only algorithm that enjoys a sufficiently general theory for problems in data science. 

\section{Numerical Experiments on GEE Problems} \label{sec:Results-gee}
One important characteristic of our general framework (see \cref{alg-general-algorithm}) that was showcased in \cref{sec:Results-cutest} is the economical use of objective function evaluations. This makes our method the preferred candidate for data science applications where accurate objective evaluations are needed, yet are expensive to obtain, and gradient evaluations are inexpensive \cite[see][]{berahas2021global,gratton2023ComplexityOfFirstOrderOFFO}.

While the above scenario might seem atypical for standard optimization problems, an important data science application with these characteristics is estimating parameters of generalized estimating equations (GEEs). This method is used when data exhibits complex grouped structure, as in repeated measurements in biomedical studies \cite[see][Chapter 3]{lipsitz2008geelongitudinal}. In its most basic form, this problem involves computing a root of
\begin{equation} \label{eq:general-gee}
    \dot{F}(\theta) = -\sum_{i=1}^n D_i(\theta)^\intercal V_i(\theta)^{-1}(Y_i-\mu_i(\theta)),
\end{equation} 
where $Y_i \in \mathbb{R}^{m_i}$ is a vector of $m_i$ measurements from group $i$; $\mu_i(\theta) \in \mathbb{R}^{m_i}$ is a model for the mean of $Y_i$ with respect to an unknown parameter $\theta \in \mathbb{R}^n$; $V_i(\theta) \in \mathbb{R}^{m_i \times m_i}$ is a model for the covariance of $Y_i$; and $D_i(\theta)$ is the Jacobian of $\mu_i(\theta)$ with respect to $\theta$.

For statistical reasons however \cite[see][Chapter 9]{mccullagh1989glm}, finding any root of \cref{eq:general-gee} is generally not enough, and a minimizer of an objective (when it exists) is desired. One way of formulating an objective is through path integrals: given some reference value $\theta_{\reference}$ and path $C \subset \mathbb{R}^n$, with smooth parametrization $p(t) : [0,1] \to \mathbb{R}^{n}$, $p(0) = \theta_{\reference}$ and $p(1) = \theta$, the objective can be defined as
\begin{equation} \label{eq:general-gee-path-form}
    F(\theta) = \int_C \dot{F} \cdot dp(t) = \int_0^1 \dot{F}(p(t))^\intercal \dot{p}(t)dt.
\end{equation}
Provided $\dot{F}$ is an irrotational vector field, the integral will be path independent, making $F(\theta)$ well-defined and ``act'' like a likelihood function. Unfortunately, in many situations \cref{eq:general-gee-path-form} might require expensive approximation techniques to evaluate. Therefore, this important and popular data science technique exhibits expensive objective evaluations yet cheap gradient evaluations, making our algorithm a prime choice because of its economical use of the objective. To illustrate this, we now present two GEE examples, and compare the numerical performance of our algorithm against optimization and root finding methods. The details of the experiments are in \cref{table:gee-expr}.

\begingroup
    \setlength{\tabcolsep}{10pt} 
    \renewcommand{\arraystretch}{1.5}
    \begin{table}[!ht]
        \centering
        \footnotesize
        \caption{GEE numerical experiment overview.}
        \label{table:gee-expr}
        \begin{tabular}{p{4cm} p{9cm}} \toprule
            \textbf{Problems} & Wedderburn's Leaf Blotch Model \cref{eq:wedderburn-obj}, and simplified Fieller-Creasy estimation \cref{eq:fieller-creasy-obj}. \\
            \textbf{Algorithms} & Same as \cref{table:cutest-exper} with the addition of Root Finding - Armijo, and Powells Dogleg \cite[Chapter 11]{nocedal2006numerical}.\\
            \textbf{Starting Points} & Wedderburn's Example: Initial components of $\theta$ are randomly generated between $-1$ and $1$, except the first and eleventh were set to $0$. Fieller-Creasy Example: Initial $\theta$ uniformly generated in $[0,1]$.\\
            \textbf{Trials} & We randomly generate a set of $100$ starting points as described, and run each algorithm once at every point.\\
            \textbf{Termination} & $1{,}000$ iterations, or gradient tolerance of $10^{-5}$. \\
            \textbf{Data Recorded} & Number of objective and gradient evaluations, and CPU Time. \\\bottomrule
        \end{tabular}
    \end{table}
\endgroup

To summarize our experiment, for each of our GEE problems, we run our method, two first order methods, three second order methods, and two root finding methods on a set of $100$ randomly generated starting points for a maximum of $1{,}000$ iterations, or until a gradient tolerance of $10^{-5}$ is attained (a successful termination).
To compare the methods, we record the number of objective and gradient evaluations, along with the CPU time until termination.

\paragraph{Wedderburn's Leaf Blotch Example} First introduced in his seminal paper to analyze a leaf disease occurring in barley plants \cite[see][]{wedderburn1974quasilikelihood}, the Leaf Blotch estimating equations and resulting optimization problem are defined for $\theta \in \mathbb{R}^{20}$ as
\begin{equation}\label{eq:wedderburn-obj}
    \dot{F}(\theta) = -\sum_{i=1}^{90} \frac{y_i-\mu_i(\theta)}{\mu_i(\theta)(1-\mu_i(\theta))} x_i, \text{\quad}
    \min_{\theta\in\mathbb{R}^{20}} F(\theta) \defeq \min_{\theta\in\mathbb{R}^{20}} \int_0^1 \dot{F}(p_{\theta}(t))^\intercal (\theta-\theta^*)dt.
\end{equation}
Here $\mu_i(\theta) = \exp(\theta^\intercal x_i)/(1 + \exp(\theta^\intercal x_i))$,\footnote{$x_i$ are covariates corresponding to $y_i$.} and $p_{\theta}(t) = t\theta + (1-t)\theta^*, ~\forall \theta \in \mathbb{R}^{20}$, where $\theta^*$ is the true minimizer.
\begin{remark}
    Wedderburn's example has an objective functions that can be found in closed form; however, for illustration, we use quadrature to approximate \cref{eq:wedderburn-obj}.
\end{remark}

In \cref{fig:wedderburn-cpu-time-results}, the CPU Time for each algorithm is presented; in \cref{fig:wedderburn-evaluation-boxplots}, we compare number of oracle evaluations for optimization algorithms. In both graphs, for each method, we only take into account trials where termination occurred by satisfaction of the gradient tolerance condition (i.e., successful termination). Note, root finding by Armijo line search is missing from the plots because more than $10{,}000$ iterations were needed to reach the gradient tolerance condition.

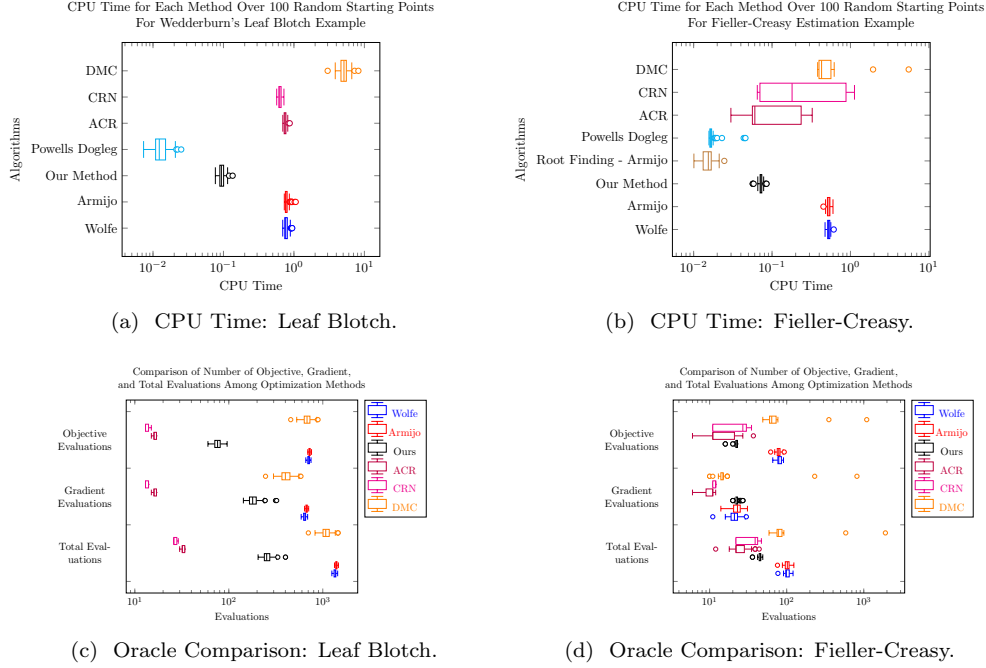
\begin{figure}
    \centering
    \begin{subfigure}{.4\textwidth}
        \begin{tikzpicture}[scale=.5]

    \pgfplotstableread[col sep=comma]{data/gee/leaf_cpu_data.csv}\cpudata


    \begin{axis}
        [
        xmode = log,
        ytick={1,2,3,4,5,6,7},
        yticklabels={Wolfe, Armijo, Our Method, Powells Dogleg, ACR, CRN, DMC},
        title style={align=center},
        title = CPU Time for Each Method Over 100 Random Starting Points \\ For Wedderburn's Leaf Blotch Example,
        ylabel = Algorithms,
        xlabel = CPU Time
        ]
        \addplot[boxplot, draw = blue, fill = white] table [y=wolfe] {\cpudata};
        \addplot[boxplot, draw = red, fill = white] table [y=armijo] {\cpudata};
        \addplot[boxplot, draw = black, fill = white] table [y=GD Triggers] {\cpudata};
        \addplot[boxplot, draw = cyan, fill = white] table [y=powells dogleg] {\cpudata};
        \addplot[boxplot, draw = purple, fill = white] table [y=Adaptive Cubic Regularization] {\cpudata};
        \addplot[boxplot, draw = magenta, fill = white] table [y=Cubic Regularized Newton] {\cpudata};
        \addplot[boxplot, draw = orange, fill = white] table [y=Dynamic Method: Curtis and Robinson] {\cpudata};
    \end{axis}
\end{tikzpicture}
        \caption{\label{fig:wedderburn-cpu-time-results} CPU Time: Leaf Blotch.}
    \end{subfigure}
    \begin{subfigure}{.4\textwidth}
        \begin{tikzpicture}[scale=.5]

    \pgfplotstableread[col sep=comma]{data/gee/fieller_cpu_data.csv}\cpudata


    \begin{axis}
        [
        xmode = log,
        ytick={1,2,3,4,5,6,7,8},
        yticklabels={Wolfe, Armijo, Our Method, Root Finding - Armijo, Powells Dogleg, ACR, CRN, DMC},
        title style={align=center},
        title = CPU Time for Each Method Over 100 Random Starting Points \\ For Fieller-Creasy Estimation Example,
        ylabel = Algorithms,
        xlabel = CPU Time
        ]
        \addplot[boxplot, draw = blue, fill = white] table [y=wolfe] {\cpudata};
        \addplot[boxplot, draw = red, fill = white] table [y=armijo] {\cpudata};
        \addplot[boxplot, draw = black, fill = white] table [y=GD Triggers] {\cpudata};
        \addplot[boxplot, draw = brown, fill = white] table [y=armijo_root] {\cpudata};
        \addplot[boxplot, draw = cyan, fill = white] table [y=powells dogleg] {\cpudata};
        \addplot[boxplot, draw = purple, fill = white] table [y=Adaptive Cubic Regularization] {\cpudata};
        \addplot[boxplot, draw = magenta, fill = white] table [y=Cubic Regularized Newton] {\cpudata};
        \addplot[boxplot, draw = orange, fill = white] table [y=Dynamic Method: Curtis and Robinson] {\cpudata};
    \end{axis}
\end{tikzpicture}
        \caption{\label{fig:fieller-creasy-cpu} CPU Time: Fieller-Creasy.}
    \end{subfigure}
    \begin{subfigure}{.4\textwidth}
        \quad\hspace{1pt}
        \begin{tikzpicture}[scale=.4]

    \pgfplotstableread[col sep=comma]{data/gee/total_evals_leaf.csv}\totalevals
    \pgfplotstableread[col sep=comma]{data/gee/func_evals_leaf.csv}\funcevals
    \pgfplotstableread[col sep=comma]{data/gee/grad_evals_leaf.csv}\gradevals

    \begin{axis}[
        legend entries = {A,B,C,D,E,F},
        title style={align=center},
        title = {Comparison of Number of Objective, Gradient, \\ and Total Evaluations Among Optimization Methods},
        xmode = log,
        boxplot/draw direction=x,
        xlabel={Evaluations},
        height=8cm,
        boxplot,
        ytick={0,1,2,...,10},
        y tick label as interval,
        yticklabels={%
            {Total Evaluations},%
            {Gradient Evaluations},%
            {Objective Evaluations},%
        },
        y tick label style={
                text width=2.5cm,
                align=center
        },
        cycle list={{blue},{red},{black},{purple},{magenta},{orange}},
        legend pos=outer north east,
        ]
        \setcounter{iloop}{0}
        \foreach \Color/\Text in {blue/Wolfe,red/Armijo,black/Ours,purple/ACR,magenta/CRN,orange/DMC}
        {\edef\temp{\noexpand\addlegendimage{boxplot legend=\Color}
        \noexpand\addlegendentry[\Color]{\Text}}
        \temp}
        
        \addplot[boxplot={draw position = .142, box extend = .125}, draw = blue, fill = white] table [y="wolfe"] {\totalevals};
        \addplot[boxplot={draw position = .2857, box extend = .125}, draw = red, fill = white] table [y="armijo"] {\totalevals};
        \addplot[boxplot={draw position = .42857, box extend = .125}, draw = black, fill = white] table [y="GD Triggers"] {\totalevals};
        \addplot[boxplot={draw position = .5714, box extend = .125}, draw = purple, fill = white] table [y="Adaptive Cubic Regularization"] {\totalevals};
        \addplot[boxplot={draw position = .7143, box extend = .125}, draw = magenta, fill = white] table [y="Cubic Regularized Newton"] {\totalevals};
        \addplot[boxplot={draw position = .85714, box extend = .125}, draw = orange, fill = white] table [y="Dynamic Method: Curtis and Robinson"] {\totalevals};

        \addplot[boxplot={draw position = 2.142,box extend = .125},draw = blue, fill = white] table [y="wolfe"] {\funcevals};
        \addplot[boxplot={draw position = 2.2857,box extend = .125}, draw = red, fill = white] table [y="armijo"] {\funcevals};
        \addplot[boxplot={draw position = 2.42857,box extend = .125}, draw = black, fill = white] table [y="GD Triggers"] {\funcevals};
        \addplot[boxplot={draw position = 2.5714, box extend = .125}, draw = purple, fill = white] table [y="Adaptive Cubic Regularization"] {\funcevals};
        \addplot[boxplot={draw position = 2.7143, box extend = .125}, draw = magenta, fill = white] table [y="Cubic Regularized Newton"] {\funcevals};
        \addplot[boxplot={draw position = 2.85714, box extend = .125}, draw = orange, fill = white] table [y="Dynamic Method: Curtis and Robinson"] {\funcevals};

        \addplot[boxplot={draw position = 1.142,box extend = .125},draw = blue, fill = white] table [y="wolfe"] {\gradevals};
        \addplot[boxplot={draw position = 1.2857,box extend = .125}, draw = red, fill = white] table [y="armijo"] {\gradevals};
        \addplot[boxplot={draw position = 1.42857,box extend = .125}, draw = black, fill = white] table [y="GD Triggers"] {\gradevals};
        \addplot[boxplot={draw position = 1.5714, box extend = .125}, draw= purple, fill = white] table [y="Adaptive Cubic Regularization"] {\gradevals};
        \addplot[boxplot={draw position = 1.7143, box extend = .125}, draw = magenta, fill = white] table [y="Cubic Regularized Newton"] {\gradevals};
        \addplot[boxplot={draw position = 1.85714, box extend = .125}, draw = orange, fill = white] table [y="Dynamic Method: Curtis and Robinson"] {\gradevals};
    \end{axis}

\end{tikzpicture}
        \caption{\label{fig:wedderburn-evaluation-boxplots} Oracle Comparison: Leaf Blotch.}
    \end{subfigure}
    \begin{subfigure}{.4\textwidth}
        \quad\quad\quad
        \begin{tikzpicture}[scale=.4]

    \pgfplotstableread[col sep=comma]{data/gee/total_evals_fieller.csv}\totalevals
    \pgfplotstableread[col sep=comma]{data/gee/func_evals_fieller.csv}\funcevals
    \pgfplotstableread[col sep=comma]{data/gee/grad_evals_fieller.csv}\gradevals

    \begin{axis}[
        legend entries = {A,B,C,D,E,F},
        title style={align=center},
        title = {Comparison of Number of Objective, Gradient, \\ and Total Evaluations Among Optimization Methods},
        xmode = log,
        boxplot/draw direction=x,
        xlabel={Evaluations},
        height=8cm,
        boxplot,
        ytick={0,1,2,...,10},
        y tick label as interval,
        yticklabels={%
            {Total Evaluations},%
            {Gradient Evaluations},%
            {Objective Evaluations},%
        },
        y tick label style={
                text width=2.5cm,
                align=center
        },
        cycle list={{blue},{red},{black},{purple},{magenta},{orange}},
        legend pos=outer north east,
        ]
        \setcounter{iloop}{0}
        \foreach \Color/\Text in {blue/Wolfe,red/Armijo,black/Ours,purple/ACR,magenta/CRN,orange/DMC}
        {\edef\temp{\noexpand\addlegendimage{boxplot legend=\Color}
        \noexpand\addlegendentry[\Color]{\Text}}
        \temp}
        
        \addplot[boxplot={draw position = .142, box extend = .125}, draw = blue, fill = white] table [y="wolfe"] {\totalevals};
        \addplot[boxplot={draw position = .2857, box extend = .125}, draw = red, fill = white] table [y="armijo"] {\totalevals};
        \addplot[boxplot={draw position = .42857, box extend = .125}, draw = black, fill = white] table [y="GD Triggers"] {\totalevals};
        \addplot[boxplot={draw position = .5714, box extend = .125}, draw = purple, fill = white] table [y="Adaptive Cubic Regularization"] {\totalevals};
        \addplot[boxplot={draw position = .7143, box extend = .125}, draw = magenta, fill = white] table [y="Cubic Regularized Newton"] {\totalevals};
        \addplot[boxplot={draw position = .85714, box extend = .125}, draw = orange, fill = white] table [y="Dynamic Method: Curtis and Robinson"] {\totalevals};

        \addplot[boxplot={draw position = 2.142,box extend = .125},draw = blue, fill = white] table [y="wolfe"] {\funcevals};
        \addplot[boxplot={draw position = 2.2857,box extend = .125}, draw = red, fill = white] table [y="armijo"] {\funcevals};
        \addplot[boxplot={draw position = 2.42857,box extend = .125}, draw = black, fill = white] table [y="GD Triggers"] {\funcevals};
        \addplot[boxplot={draw position = 2.5714, box extend = .125}, draw = purple, fill = white] table [y="Adaptive Cubic Regularization"] {\funcevals};
        \addplot[boxplot={draw position = 2.7143, box extend = .125}, draw = magenta, fill = white] table [y="Cubic Regularized Newton"] {\funcevals};
        \addplot[boxplot={draw position = 2.85714, box extend = .125}, draw = orange, fill = white] table [y="Dynamic Method: Curtis and Robinson"] {\funcevals};

        \addplot[boxplot={draw position = 1.142,box extend = .125},draw = blue, fill = white] table [y="wolfe"] {\gradevals};
        \addplot[boxplot={draw position = 1.2857,box extend = .125}, draw = red, fill = white] table [y="armijo"] {\gradevals};
        \addplot[boxplot={draw position = 1.42857,box extend = .125}, draw = black, fill = white] table [y="GD Triggers"] {\gradevals};
        \addplot[boxplot={draw position = 1.5714, box extend = .125}, draw=purple, fill = white] table [y="Adaptive Cubic Regularization"] {\gradevals};
        \addplot[boxplot={draw position = 1.7143, box extend = .125}, draw = magenta, fill = white] table [y="Cubic Regularized Newton"] {\gradevals};
        \addplot[boxplot={draw position = 1.85714, box extend = .125}, draw = orange, fill = white] table [y="Dynamic Method: Curtis and Robinson"] {\gradevals};
    \end{axis}

\end{tikzpicture}
        \caption{\label{fig:fieller-creasy-evaluations} Oracle Comparison: Fieller-Creasy.}
    \end{subfigure}
    \caption{CPU time and oracle evaluations between our method and other algorithms for Wedduerburn's example and the Fieller-Creasy estimation problem.}
\end{figure}

From \cref{fig:wedderburn-cpu-time-results}, DMC takes the longest; a surprising cluster of equally performing methods --- CRN, ACR, and Gradient Descent using Armijo and Wolfe line search --- are faster than DMC; our procedure is faster than all those using objective function information; and Powell's Dogleg method, which does not require any objective function information, is the quickest. 
For first-order methods, these relative CPU times are readily explained by the number of objective evaluations required (see \cref{fig:wedderburn-evaluation-boxplots}); whereas, for second-order methods, the slower CPU times (despite fewer evaluations) are caused by expensive sub-problem solvers. For this example, Powell's dogleg method seems to be the best choice, followed by our procedure; however, as we will see in the next example, Powell's dogleg method often finds roots corresponding to maximizers, whereas our solutions correspond to the minimizer.


\paragraph{Fieller-Creasy Ratio Estimation Example} We now consider the simplified Fieller-Creasy problem \cite[see][Chapter 9, \S 4]{mccullagh1989glm}. The estimating equation and resulting objective function (by the fundamental theorem of calculus) for $\theta \in \mathbb{R}$, provided $N \in \mathbb{N}$ datapoints, $\{(Y_{i1}, Y_{i2})\}_{i=1}^N$, are
\begin{equation}\label{eq:fieller-creasy-obj}
    \dot{F}(\theta) = -\sum_{i=1}^N \frac{(Y_{i2}+\theta Y_{i1})(Y_{i1} - \theta Y_{i2})}{\sigma^2(1+\theta^2)^2} \text{\quad and \quad} \min_{\theta\in\mathbb{R}} F(\theta) \defeq \min_{\theta\in\mathbb{R}} \int_{0}^\theta \dot{F}(t) dt.
\end{equation}

\begin{remark}
    Fieller-Creasy has an objective functions that can be found in closed form; however, for illustration, we use quadrature to approximate \cref{eq:fieller-creasy-obj}.
\end{remark}

We use the same experimental design as in \cref{table:gee-expr}, and use a simulated dataset of 50 points, where $\{(Y_{i1}, Y_{i2})\}_{i=1}^{50}$ are independent, normally distributed random variables, with means $\{(\mu_i, \mu_i/5)\}_{i=1}^{50}$, and standard deviation of $.05$. The means are generated by taking $50$ points evenly spaced between $1$ and $3$. Generating the random datapoints in this way, there will be two stationary points, one at $\approx 5$ and another at $\approx -0.2$, the first corresponding to a minimum and the second corresponding to a maximum.

Finally, CPU times are plotted in \cref{fig:fieller-creasy-cpu}, and the number of oracle evaluations are plotted for the optimization method in \cref{fig:fieller-creasy-evaluations}. In both plots and for each algorithm, we only include counts where the algorithm had a successful termination event, and were close to the approximate minimizer.\footnote{The criterion for an approximate minimizer being successful termination, and the terminal iterate between $4.9$ and $5$. For an approximate maximizer, the criterion having successful termination, and the terminal iterate between $-.2$ and $-.21$.\label{fn-min-max-def}} Note, that root finding using Armijo line search is now present in the comparison as well.

In \cref{fig:fieller-creasy-cpu}, as before we see root finding methods are the fastest, followed now by adaptive cubic regularization and our method, then trailed by the remaining first-order and second-order algorithms. We remark here that ACR has an average time that is \emph{marginally} faster than our method, however this is a one dimensional problem, and as we have seen with the previous example the subproblems can get expensive as the dimension grows. Again, the root finding methods would seem to be the more favorable choice as they are fastest; however, the root finding methods are less reliable in finding a minimizer. Specifically, in \cref{table:gee-terminal-iterate-state}, we count the number of times a solver approximately finds a minimizer, a maximizer, or neither.\cref{fn-min-max-def} We readily observe that the root-finding methods tend to be substantially less reliable in comparison to our procedure. Hence, despite their faster speed, our procedure is favorable in both speed and reliability.



\begin{table}[H]
    \centering
    \footnotesize
    \caption{Categorization of Terminal Points}
    \label{table:gee-terminal-iterate-state}
    \begin{tabular}{c c c c} \toprule
        \textbf{Algorithm} & \textbf{Approximate Minimizer} & \textbf{Approximate Maximizer} & \textbf{Neither} \\\midrule
        DMC & 79 & 0 & 21\\
        CRN & \textbf{100} & 0 & 0\\
        ACR & 96 & 0 & 4\\
        Powell's Dogleg & 59 & 41 & 0 \\
        Root Finding - Armijo & 39 & 16 & 45 \\
        Our Method & \textbf{100} & 0 & 0\\
        GD with Armijo & 27 & 0 & 73\\
        GD with Wolfe & 23 & 0 & 77 \\\bottomrule
    \end{tabular}
\end{table}

\paragraph{In Summary} Using two GEE examples, we have compared our procedure against two first order optimization algorithms, three second order optimization algorithms, and two root finding methods. In comparison to the optimization methods, our procedure is faster and just as reliable in finding a local minimizer. In comparison to the root finding problems, our procedure is slower but is more reliable in finding a local minimizer. Thus, for the needs of such data science problems, our procedure provides the fastest and most reliable solver. 

\section{Conclusion}  \label{sec:Conclusion}
We remarked that many contemporary and traditional optimization methods face challenges for optimization problems arising in data science \cite[see][]{varner2024challenges}.
To address this gap, we presented a new, general optimization methodology (\cref{alg-general-algorithm}) that is better suited to such problems. We proved global convergence results (see \cref{result-gradient}) under reasonable assumptions for data science scenarios. Furthermore, we specialized this framework, and developed a novel step size procedure using negative gradient directions (\cref{alg:novel-step-size}) that is not only extremely competitive on general optimization problems from the CUTEst library (\cref{sec:Results-cutest}), but is also superior to the alternatives on target data science applications (see \cref{sec:Results-gee}). Several open questions arise from this work. Methodologically, we hope to extend our general framework to develop novel step size routines under realistic assumptions for other optimization approaches (e.g., coordinate descent). Theoretically, we hope to provide local convergence results for \cref{alg-general-algorithm,alg:novel-step-size} for different step direction choices; study how the gradient behaves around stationary points; and study the divergence regime. Lastly, in the application setting, it would be interesting to scale our algorithm to large instances of estimating equation problems to develop new insights to applied questions.

\bibliographystyle{siamplain}
\bibliography{references}

\end{document}